\documentclass[11pt,reqno]{amsart}

\usepackage{amsmath,amssymb,amsthm,xcolor,tikzsymbols}
\usepackage[margin=1.5in]{geometry}
\usepackage{setspace}
\usepackage[colorlinks=True,citecolor=blue,linkcolor=red]{hyperref}

\usepackage[vcentermath]{genyoungtabtikz}
\Ylinethick{.9pt}
\Yboxdim{14pt}
\usepackage{ytableau}

\newcommand{\absval}[1]{\left\lvert #1 \right\rvert}
\newcommand{\bon}{\overline{1}}
\newcommand{\B}{\mathcal{B}}
\newcommand{\C}{\mathcal{C}}
\newcommand{\ch}{\operatorname{ch}}
\newcommand{\g}{\mathfrak{g}}
\newcommand{\gl}{\mathfrak{gl}}

\newcommand{\iso}{\cong}

\newcommand{\longhookarrow}{\lhook\joinrel\longrightarrow}
\renewcommand{\mid}{:}

\newcommand{\q}{\mathfrak{q}}

\newcommand{\R}{\mathcal{R}}

\newcommand{\rd}{\mathrm{read}}  
\newcommand{\red}{\overline{\rd}}  
\newcommand{\SDT}{\mathrm{SDT}}

\newcommand{\T}{\mathcal{T}}
\newcommand{\wt}{\mathrm{wt}}
\newcommand{\zero}{\boldsymbol{0}}
\newcommand{\ZZ}{\mathbf{Z}}

\newcommand{\fb}{\Yfillcolour{blue!10}}  
\newcommand{\fe}{\Yfillcolour{white}}  
\newcommand{\ml}[1]{\color{darkred} \mathbf{#1}}  

\definecolor{darkred}{rgb}{0.7,0,0} 
\newcommand{\defn}[1]{{\color{darkred}\emph{#1}}} 

\theoremstyle{plain}
\newtheorem{thm}{Theorem}[section]
\newtheorem{lemma}[thm]{Lemma}

\newtheorem{prop}[thm]{Proposition}
\newtheorem{cor}[thm]{Corollary}
\theoremstyle{definition}
\newtheorem{dfn}[thm]{Definition}
\newtheorem{ex}[thm]{Example}
\newtheorem{remark}[thm]{Remark}
\numberwithin{equation}{section}
\numberwithin{figure}{section}
\numberwithin{table}{section}

\usepackage{listings}
\lstdefinelanguage{Sage}[]{Python}
{morekeywords={False,True},sensitive=true}
\usepackage[colorinlistoftodos]{todonotes}

\AtBeginDocument{%
   \def\MR#1{}
}

\begin{document}

\title[Candidate for the crystal $B(-\infty)$ for the queer Lie superalgebra]{Candidate for the crystal $B(-\infty)$\\ for the queer Lie superalgebra}

\author[B.~Salisbury]{Ben Salisbury}
\address[B.~Salisbury]{Department of Mathematics, Central Michigan University, Mount Pleasant, MI 48859, USA}
\email{ben.salisbury@cmich.edu}
\urladdr{http://people.cst.cmich.edu/salis1bt/}
\thanks{B.S.\ was partially supported by Simons Foundation grant 429950.}

\author[T.~Scrimshaw]{Travis Scrimshaw}
\address[T.~Scrimshaw]{School of Mathematics and Physics, The University of Queensland, St. Lucia, QLD 4072, Australia}
\email{tcscrims@gmail.com}
\urladdr{https://people.smp.uq.edu.au/TravisScrimshaw/}
\thanks{T.S.\ was partially supported by Australian Research Council DP170102648.}

\begin{abstract}
It is shown that the direct limit of the semistandard decomposition tableau model for polynomial representations of the queer Lie superalgebra exists, which is believed to be the crystal for the upper half of the corresponding quantum group. An extension of this model to describe the direct limit combinatorially is given. Furthermore, it is shown that the polynomial representations may be recovered from the limit in most cases.
\end{abstract}

\maketitle

\section{Introduction}

In the 1990s, Kashiwara began the study of crystals, a combinatorial skeleton of a quantum group representation $U_q(\g)$, where $\g$ is a symmetrizable Kac--Moody algebra.  (See also Lusztig's canonical basis \cite{Lusztig90} and Littelmann's path model \cite{L94,L95-2}.)  Kashiwara showed~\cite{K90,K91} that irreducible highest weight representations have crystal bases $B(\lambda)$ and that the lower half of the quantum group has a crystal basis $B(\infty)$. The global crystal basis (general $q$ version of a crystal basis) is known~\cite{GL93} to equal the canonical basis introduced by Lusztig~\cite{Lusztig90}.
Kashiwara also proved that the direct limit of $B(\lambda)$ is isomorphic to $B(\infty)$ and one can recover $B(\lambda)$ by cutting a part of $B(\infty)$ by taking the tensor product with a specific crystal $\R_\lambda$.  Using the direct limit, numerous combinatorial models for $B(\infty)$ have been developed such as (marginally) large tableaux~\cite{Cliff98,HL08} and rigged configurations~\cite{SalisburyS15,SalisburyS15II,SalisburyS16II}.

For Lie superalgebras, there are two natural analogs of $\gl(n)$.  The first is the general linear Lie superalgebra $\gl(m|n)$, where crystal bases have been constructed for the polynomial representations~\cite{BKK00}, Kac modules~\cite{Kwon14}, and $B(\infty)$ for $\gl(m|1)$~\cite{Clark16}.  Furthermore, the character theory for $\gl(m|n)$ has been well-studied~\cite{Brundan03,CHR15,Serganova93,SZ07,VHKT90} with connections to quasisymmetric functions~\cite{Kwon09}. The other is the queer superalgebra $\q(n)$.  The tensor powers of the fundamental representation form a semisimple category~\cite{GJKK10}, the irreducible representations are called the polynomial representations, and crystal bases of these irreducible representations have been constructed using semistandard decomposition tableaux~\cite{GJKKK14,GJKKK15}. The character theory of $\q(n)$ has also been studied~\cite{Brundan04,CKW17,SZ15}. In particular, the characters of the polynomial representations are Schur $P$-functions, which (along with the closely related Schur $Q$-functions) are of broad interest with other crystal connections~\cite{BH95,HMP17,HPS17,Jozefiak91,Schur11}. Recently, a local characterization of the crystals for polynomial representations of $\q(n)$ was given~\cite{AKO18,GHPS18} in analogy to the Stembridge axioms~\cite{Stem.2003}.

The goal of this paper is to construct the direct limit of the crystals of polynomial representations. We call this limit $B(-\infty)$ as we are considering the polynomial representations as \emph{lowest} weight representations and taking the corresponding limit. We believe this to be the crystal basis of the contragredient dual of the $q$-Weyl module $W(0)$ defined in~\cite[Section~4]{GJKK10},\footnote{The module as $q \to 1$ was called a Verma module in~\cite[Section~2.1.6]{CW12}.} in parallel to the case for $\gl(n)$ where $B(-\infty)$ is the crystal basis of the contragredient dual of the Verma module $M(0)$.
Subsequently, this should be a combinatorial model for the crystal basis of upper half of the corresponding quantum group $U_q\bigl( \q(n) \bigr)$ and also for the contragredient dual of the $q$-Weyl module $W(\lambda)$ after shifting the weight by $\lambda$.  This claim is partially supported by a character calculation and an explicit combinatorial characterization of the lowest weight elements in $B(-\infty)$. (See Proposition \ref{prop:char} and Proposition~\ref{prop:alg}.)

Our method follows the construction of (marginally) large tableaux for semistandard decomposition tableaux $\SDT(\lambda)$ by showing the one can construct a directed system and the $\q(n)$-crystal operators respect enlarging the shape.  (See Lemma \ref{lem:up_inclusions} and Corollary \ref{cor:up_system}.)  We then identify elements in each $\SDT(\lambda)$ based on their distance from the lowest weight element and take a distinguished representative.  Thus, our model is this limit crystal $\SDT(-\infty)$ of distinguished representatives that we call dual marginally large semistandard decomposition tableaux.  (See Theorem \ref{thm:up-inf}.)  Our other main result (see Theorem \ref{thm:cutting_it_out}) is describing how we can recover $\SDT(\lambda)$ from $\SDT(-\infty)$ using a dual version of $\R_{\lambda}$ in the case when $\lambda$ corresponds to a strict partition that has maximal length. We expect this to generalize to the case when $\lambda$ does not have maximal length by a modification of the tensor product rule, which we also expect to construct crystals for dual polynomial representations.

It is noteworthy that one cannot take the limit of the polynomial representations considered as highest weight representations. Indeed, if we consider the shapes $\lambda = (5,3,1)$ and $\mu = (6,4,1)$, then for the direct limit, we must have an inclusion $B(\lambda) \lhook\joinrel\longrightarrow B(\mu)$. However, if $T^{\lambda} \in B(\lambda)$ and $T^{\mu} \in B(\mu)$ are the unique highest weight elements of weight $\lambda$ and $\mu$, respectively, then
\[
f_{\bon} f_2 f_{\bon} f_2 f_2 T^{\lambda} \neq \zero,
\hspace{70pt}
f_{\bon} f_2 f_{\bon} f_2 f_2 T^{\mu} = \zero.
\]
Roughly speaking, the reason this inclusion fails is because adding entries to go from $T^{\lambda}$ to $T^{\mu}$ does not behave well with respect to the crystal operators $e_{\bon}$ and $f_{\bon}$. Indeed, we have
\begin{align*}
T^{\lambda} & = \gyoung(32211,:;211,::;1)\,,
&
f_2 f_{\bon} f_2 f_2 T^{\lambda} & = \gyoung(33313,:;211,::;1)\,,
\\
T^{\mu} & = \gyoung(322211,:;2111,::;1)\,,
&
f_2 f_{\bon} f_2 f_2 T^{\mu} & = \gyoung(333312,:;2111,::;1)\,,
\end{align*}
and $f_{\bon}$ wants to change a $1$ to a $2$ in the rightmost entry of the topmost row. However, the $2$ that was added to the top row of $T^{\mu}$ led to an extra $2$ in the top row of $f_2 f_{\bon} f_2 f_2 T^{\mu}$, which breaks the inclusion. Thus there is no direct limit of $B(\lambda)$ as highest weight crystals, in sharp contrast to $\gl(n)$ where we can take both limits.

\section{Background}
\label{sec:background}

Partitions and tableaux will be written using English convention.  For brevity, write a column of height $h$ as
\[
\gyoung(<x_1>,<x_2>,|2\vdts,<x_h>)
= [x_1, x_2, \dotsc, x_h]^\top
\]

\subsection{Crystals for the superalgebra $\q(n)$}

Let $I_0 = \{1,\dots,n-1\}$ and $I = I_0 \sqcup \{\overline1\}$.  Denote the standard basis vectors of $\ZZ^n$ by $\epsilon_1,\dots,\epsilon_n$ and define the \defn{simple root} $\alpha_i = \epsilon_i - \epsilon_{i+1}$ for each $i \in I_0$.
Let $Q = \bigoplus_{i \in I_0} \ZZ \alpha_i$ be the root lattice and $Q^+ = \bigoplus_{i \in I_0} \ZZ_{\geq 0} \alpha_i$ and $Q^- = \bigoplus_{i \in I_0} \ZZ_{\leq 0} \alpha_i$ be the positive and negative root cones.
Let $\Phi^{\pm} = \{ \pm(\epsilon_i - \epsilon_j) \mid 1 \leq i < j \leq n \}$ be the set of positive and negative roots.
Set
\[
\Lambda^- = \left\{ \lambda = -\lambda_1\epsilon_1 - \cdots - \lambda_n\epsilon_n \in \ZZ^n_{\le0} : \begin{gathered} \lambda_i \ge \lambda_{i+1} \text{ and } \lambda_i = \lambda_{i+1} \text{ implies } \\ \lambda_i = \lambda_{i+1} = 0 \text{ for all } i=1,\dots,n-1 \end{gathered} \right\}.
\]
Equip $\Lambda^-$ with a partial order $\lambda \le \mu$ if and only if $\mu-\lambda \in \Lambda^-$.\footnote{This is not the usual (opposite) dominance order, where the result should be in $Q^-$, but instead should be thought of as the lex order.}
An element $\lambda = -\lambda_1\epsilon_1 - \cdots - \lambda_n\epsilon_n \in \Lambda^-$ will be henceforth be identified with the strict partition $w_0\lambda = (\lambda_n,\dots,\lambda_1)$.  

\begin{dfn}[{\cite[Definition 1.2.1]{K93}}]
An \defn{abstract $\gl(n)$-crystal} is a set $\B$ together with maps $e_i,f_i \colon \B \longrightarrow \B\sqcup\{\zero\}$, $\varphi_i,\varepsilon_i\colon \B \longrightarrow \ZZ\sqcup\{-\infty\}$, for $i \in I_0$, and $\wt\colon \B \longrightarrow \ZZ^n$ satisfying the following conditions:
\begin{enumerate}
\item for any $i \in I_0$ such that $e_ib \neq \zero$, we have $\wt(e_ib) = \wt(b) + \alpha_i$;
\item for any $i \in I_0$ such that $f_ib \neq \zero$, we have $\wt(f_ib) = \wt(b) - \alpha_i$;
\item \label{axiom:phi_ep_wt} for any $i \in I_0$ and $b\in \B$, we have $\varphi_i(b) = \varepsilon_i(b) + \wt_i(b)$;
\item for any $i \in I_0$ and $b,b' \in \B$, we have $f_ib=b'$ if and only if $b = e_ib'$;
\item for any $i \in I_0$ and $b\in \B$ such that $e_ib \neq \zero$, we have $\varepsilon_i(e_ib) = \varepsilon_i(b) - 1$ and $\varphi_i(e_ib) = \varphi_i(b)+1$;
\item for any $i \in I_0$ and $b \in \B$ such that $f_ib \neq \zero$, we have $\varepsilon_i(f_ib) = \varepsilon_i(b)+1$ and $\varphi_i(f_ib) = \varphi_i(b) - 1$;
\item for any $i \in I_0$ and $b\in \B$ such that $\varphi_i(b) = -\infty$, we have $e_ib = f_ib = \zero$.
\end{enumerate}
\end{dfn}
In the above definition (and throughout), the notation $\wt_i(b)$ means $\wt_i(b) = \mu_i - \mu_{i+1}$ provided $\wt(b) = \sum_{i \in I_0} \mu_i \epsilon_i$.

\begin{dfn}[{\cite[Definition 1.9]{GJKKK15}}]
An \defn{abstract $\q(n)$-crystal} is an abstract $\gl(n)$-crystal $\B$ together with maps $e_{\bon},f_{\bon} \colon \B \longrightarrow \B\sqcup\{\zero\}$ such that
\begin{enumerate}
\item $\wt(\B) \subset \ZZ^n_{\ge0}$;
\item $\wt(e_{\overline{1}}b) = \wt(b) + \alpha_1$ provided $e_{\overline{1}}b \neq \zero$;
\item $\wt(f_{\overline{1}}b) = \wt(b) - \alpha_1$ provided $f_{\overline{1}}b \neq \zero$;
\item for any $b,b' \in \B$, $f_{\overline{1}}b=b'$ if and only if $b = e_{\overline{1}}b'$;
\item if $3 \le i \le n-1$, we have
\begin{enumerate}
\item the operators $e_{\overline{1}}$ and $f_{\overline{1}}$ commute with $e_i$ and $f_i$, and
\item if $e_{\overline{1}}b \in \B$, then $\varepsilon_i(e_{\overline{1}}b) = \varepsilon_i(b)$ and $\varphi_i(e_{\overline{1}}b) = \varphi_i(b)$.
\end{enumerate}
\end{enumerate}
\end{dfn}

Let $\B$ and $\C$ be abstract $\q(n)$-crystals.  A \defn{crystal morphism} is a map $\psi \colon \B \longrightarrow \C \sqcup \{\zero\}$ such that
\begin{enumerate}
\item if $b \in \B$ and $\psi(b) \in \C$, then for all $i\in I_0$,
\[
\wt\bigl(\psi(b)\bigr) = \wt(b), \qquad
\varepsilon_i\bigl(\psi(b)\bigr) = \varepsilon_i(b), \qquad
\varphi_i\bigl(\psi(b)\bigr) = \varphi_i(b);
\]
\item for $b \in \B$ and $i\in I$, we have $\psi(e_i b) = e_i \psi(b)$ provided $\psi(e_ib) \neq \zero$ and $e_i\psi(b) \neq \zero$;
\item for $b\in \B$ and $i\in I$, we have $\psi(f_i b) = f_i \psi(b)$ provided $\psi(f_ib) \neq \zero$ and $f_i\psi(b) \neq \zero$.
\end{enumerate}
A morphism $\psi$ is called \defn{strict} if $\psi$ commutes with $e_i$ and $f_i$ for all $i \in I$.  Moreover, a morphism $\psi\colon \B \longrightarrow \C\sqcup\{\zero\}$ is called an \defn{embedding} (resp.\ \defn{isomorphism}) if the induced map $\psi\colon \B\longrightarrow \C \sqcup \{\zero\}$ is injective (resp.\ bijective).

Again let $\B$ and $\C$ be abstract $\q(n)$-crystals.  The \defn{tensor product} $\B \otimes \C$ is defined to be the Cartesian product $\B\times \C$ equipped with crystal operations defined, for $i\in I_0$, by
\begin{equation}\label{eq:tensor}
\begin{aligned}
e_i(b \otimes c) &= \begin{cases}
e_i b \otimes c & \text{if } \varphi_i(c) < \varepsilon_i(b), \\
b \otimes e_i c & \text{if } \varphi_i(c) \ge \varepsilon_i(b),
\end{cases} &
e_{\overline{1}}(b\otimes c) &= \begin{cases}
b \otimes e_{\overline{1}}c & \text{if } e_{\overline{1}}b = f_{\overline{1}}b = \zero, \\
e_{\overline{1}}b \otimes c & \text{otherwise}.
\end{cases}\\
f_i(b \otimes c) &= \begin{cases}
f_i b \otimes c & \text{if } \varphi_i(c) \le \varepsilon_i(b), \\
b \otimes f_i c & \text{if } \varphi_i(c) > \varepsilon_i(b),
\end{cases} & 
f_{\overline{1}}(b\otimes c) &= \begin{cases}
b \otimes f_{\overline{1}}c & \text{if } e_{\overline{1}}b = f_{\overline{1}}b = \zero, \\
f_{\overline{1}}b \otimes c & \text{otherwise}.
\end{cases}
\end{aligned}
\end{equation}

\begin{remark}
This is equivalent to the rule in~\cite[Remark~2.4]{GHPS18} in the situations we consider.  Moreover, this is the reverse convention of the tensor product to that given in~\cite{GJKKK14}.
\end{remark}

The remaining crystal structure of $\B \otimes \C$ is defined as
\begin{align*}
\varepsilon_i(b \otimes c) & = \max\bigl(\varepsilon_i(c), \varepsilon_i(b) - \wt_i(c) \bigr),
\\
\varphi_i(b \otimes c) & = \max\bigl(\varphi_i(b), \varphi_i(c) + \wt_i(b) \bigr),
\\
\wt(b \otimes c) & = \wt(b) + \wt(c).
\end{align*}

Following the method of \cite[p.~74]{K02}, one can construct direct limits in the category of abstract $\q(n)$-crystals.  Indeed, let $\{\B_j\}_{j\in J}$ be a directed system of crystals and let $\psi_{k,j}\colon \B_j \longrightarrow \B_k$, $j\le k$, be a crystal embedding (with $\psi_{j,j}$ being the identity map on $\B_j$) such that $\psi_{k,j} \psi_{j,i} = \psi_{k,i}$.  Let $\vec{\B} = \varinjlim \B_j$ be the direct limit of this system and let $\psi_j \colon \B_j \longrightarrow \vec{\B}$.  Then $\vec{\B}$ has a crystal structure induced from the crystals $\{\B_j\}_{j \in J}$.  Indeed, for $\vec{b} \in \vec{\B}$ and $i\in I$, define $e_i\vec{b}$ to be $\psi_j(e_i b_j)$ if there exists $b_j \in \B_j$ such that $\psi_j(b_j) = \vec{b}$ and $e_i(b_j) \neq 0$.  This definition does not depend on the choice of $b_j$.  If there is no such $b_j$, then set $e_i\vec{b} = 0$.  The definition of $f_i\vec{b}$ is similar.  Moreover, the functions $\wt$, $\varepsilon_i$, and $\varphi_i$ on $\B_j$ extend to functions on $\vec{\B}$.

\subsection{Semistandard decomposition tableaux}

This section summarizes the results of~\cite{GJKKK14} using the conventions of~\cite{GHPS18}.

\begin{dfn} 
Let $\eta = (\eta_n,\dots,\eta_1)$ be a strict partition.  Define $|\eta| = \eta_1 + \cdots + \eta_n$ and $\ell(\eta)$ to be the number of $1\le i \le n$ such that $\eta_i \neq 0$.
\begin{enumerate}
\item The \defn{shifted Young diagram of shape $\eta$} is an array of cells in which the $i$-th row has $\eta_{n+1-i}$ cells, and is shifted $i-1$ units to the right with respect to the top row.  
\item A word $u = u_1u_2\cdots u_N$ is a \defn{hook word} if there exists $1 \le k \le N$ such that
\[
u_1 \ge u_2 \ge \cdots \ge u_k < u_{k+1} < \cdots < u_N.
\]
\item A \defn{semistandard decomposition tableau of shifted shape $\eta$} is a filling $T$ of $\eta$ with letters from $\{1,2,\dots,n\}$ such that
\begin{enumerate}
\item the word $v_i$ formed by reading the $i$-th row from left to right is a hook word of length $\eta_{n-i+1}$, and
\item $v_i$ is a hook subword of maximal length in $v_{i+1}v_i$ for $1 \le i \le \ell(\eta)-1$, where the entries of $v_i$ are taken consecutively.
\end{enumerate}
\item Set $\rd(T)$ to be the word obtained by reading $T$ in rows from right to left starting at the top.\footnote{The choice of the reading word and the opposite order of the tensor product means we obtain the same crystal in~\cite{GJKKK14}; \textit{i.e.}, the two reversals nullify the effects of each other.} Let $\rd_i(T)$ denote the subword of $\rd(T)$ consisting of only the letters $i$ and $i+1$.
Let $\red_i(T)$ denote the subword of $\rd(T)$ formed by successively removing all pairs of letters $(i+1, i)$ (in that order).
\item For $\lambda \in \Lambda^-$, let $\SDT(\lambda)$ denote the set of all semistandard decomposition tableaux of shape $w_0\lambda$.
\end{enumerate}
\end{dfn}

Note that $\red_i(T)$ is a word consisting of $i$'s following by $(i+1)$'s.

We will require the following characterization of semistandard decomposition tableaux.

\begin{prop}[{\cite[Proposition~2.3]{GJKKK14}}]
\label{prop:decomp_characterization}
A tableau $T$ is a decomposition tableau if
\begin{enumerate}
\item every row is a hook word,
\item \label{char:leftmost_box} the leftmost entry of a given row is strictly larger than every entry in the row below it, and
\item neither of the following configurations occur:
\[
\begin{array}{c@{\hspace{50pt}}c}
\gyoung(a,c_2\hdts b) & \gyoung(b_2\hdts c,:::;a) \\[10pt]
a \leq b \leq c & a < b < c \\
\end{array}.
\]
\end{enumerate}
\end{prop}

The prohibited configurations in Proposition~\ref{prop:decomp_characterization} will be referred to as \defn{type L} and \defn{type U}, respectively.

\begin{dfn}\label{dfn:finite_operators}
Let $T$ be a semistandard decomposition tableau of shape $w_0\lambda$.
\begin{enumerate}
\item Suppose $i \in I_0$.
\begin{enumerate}
\item If there is no $i+1$ in $\red_i(T)$, then $e_iT = \zero$. Otherwise, $e_iT$ is the tableau obtained from $T$ by changing the $(i+1)$-box corresponding to the leftmost $i+1$ in the subword above to an $i$-box.
\item If there is no $i$ in $\red_i(T)$, then $f_iT = \zero$. Otherwise, $f_iT$ is the tableau obtained from $T$ by changing the $i$-box corresponding to the rightmost $i$ remaining in the subword above to an $(i+1)$-box. 
\end{enumerate}
\item Suppose $i = \bon$.
\begin{enumerate}
\item If the leftmost letter in $\rd_1(T)$ is $1$, then $e_{\bon}T =0$. Otherwise $e_{\bon}T$ is the tableau obtained from $T$ by changing the $2$-box corresponding to the leftmost $2$ in $\rd_1(T)$ to a $1$-box.
\item If the leftmost letter in $\rd_1(T)$ is $2$, then $f_{\bon}T =0$. Otherwise $f_{\bon}T$ is the tableau obtained from $T$ by changing the $1$-box corresponding to the leftmost $1$ in $\rd_1(T)$ to a $2$-box.
\end{enumerate}
\end{enumerate}
\end{dfn}

For a $\lambda \in \Lambda^-$ with $\ell(\lambda) = N$, define $L^\lambda \in \SDT(\lambda)$ to be the tableau whose $i$-th row from the bottom contains only the letter $i$.  As described in the next theorem, $L^\lambda$ is a generator of $\SDT(\lambda)$, and a generalization of $L^\lambda$ will serve as a generator for the candidate for $B(-\infty)$ in the next section.

\begin{ex}
Let $n=5$ and $w_0\lambda = (7,4,3,2,1)$.  Then
\[
L^\lambda = \gyoung(5555555,:;4444,::;333,:::;22,::::;1)\ .
\]
\end{ex}

An element $T \in \SDT(\lambda)$ is called \defn{lowest weight} (resp.\ \defn{highest weight}) if $f_i T = \zero$ (resp.~$e_iT  = \zero$) for all $i \in I$. 

\begin{thm}[{\cite[Theorem 2.5]{GJKKK14}}]
\label{thm:SDT_crystal}
For $\lambda \in \Lambda^-$, the set $\SDT(\lambda)$ together with the operators defined in Definition \ref{dfn:finite_operators} form an abstract $\q(n)$-crystal isomorphic to the crystal of the irreducible highest weight $\q(n)$-module with highest weight $w_0\lambda$.  Moreover, $\SDT(\lambda)$ is generated by $L^{\lambda}$ (hence, connected).
\end{thm}

Let us briefly remark on the notation $w_0\lambda$. Typically $w_0$ denotes the long element of the symmetric group $\mathfrak{S}_n$, which acts naturally on $\ZZ^n$. From~\cite[Theorem~2.5]{GJKKK14}, the lowest weight that appears in the character of $B(w_0\lambda)$ has weight $w_0 \cdot \lambda$, where now $w_0$ is acting on $\lambda$ as an element of the Weyl group of $\gl(n)$, up to shifting by $(1, \dotsc, 1)$. However, this does not change the structure of the crystals (up to this shift in the weight).

\section{Main results}
\label{sec:results}

In this section, we prove our main results. To do so, we first give $\q(n)$-analogs of some auxiliary crystals used in the construction of $B(\infty)$. Next, an explicit description of the direct limit $\SDT(-\infty)$ is given, and this object is our proposed $B(-\infty)$. It is then shown that $\{ \SDT(\lambda) \mid \lambda \in \Lambda^-\}$ forms a directed system and is isomorphic to $\SDT(-\infty)$.  We conclude with a method to recover $\SDT(\lambda)$ from $\SDT(-\infty)$ in some cases.

\subsection{Auxiliary crystals}

There are two abstract crystals which are necessary for later work.  The first of the two crystals defined next simply shift the weights of a given crystal when tensored, and the second will be used to construct $\SDT(\lambda)$ from $\SDT(-\infty)$.

\begin{dfn}\label{defn:aux}
Let $\lambda = (\lambda_1,\dotsc,\lambda_n) \in \ZZ^n$.
\begin{enumerate}
\item \label{item:T_crystal}

Define $\T_\lambda = \{ t_\lambda\}$ with operations
\begin{gather*}
e_it_\lambda = e_{\overline{1}}t_\lambda = f_it_\lambda = f_{\overline{1}}t_\lambda = \zero,
\\
\varepsilon_i(t_\lambda) = \varphi_i(t_\lambda) = -\infty,
\\
\wt(t_\lambda) = \lambda,
\end{gather*}
for $i\in I_0$.
\item \label{item:R_crystal}
Define $\R_\lambda^\vee = \{r_\lambda^\vee\}$ with operations
\begin{gather*}
e_ir_\lambda^\vee = e_{\overline{1}}r_\lambda^\vee = f_ir_\lambda^\vee = f_{\overline{1}}r_\lambda^\vee = \zero,
\\
\varepsilon_i(r_\lambda^\vee) = 0, 
\ 
\varphi_i(r_\lambda^\vee) = \lambda_{i}-\lambda_{i+1},
\\
\wt(r_\lambda^\vee) = \lambda,
\end{gather*}
for $i\in I_0$.
\end{enumerate}
\end{dfn}

The following is clear from the definitions.

\begin{lemma}
\label{lemma:RT_crystals}
If $\lambda \in \ZZ_{\ge0}^n$, then $\T_\lambda$ and $\R_{\lambda}^{\vee}$ are abstract $\q(n)$-crystals.
\end{lemma}

\subsection{Candidate for $B(-\infty)$}

\begin{dfn} A semistandard decomposition tableau $T$ for $\q(n)$ is called \defn{dual large} if
\begin{enumerate}
\item $T$ has $n$ rows, and
\item for all $1 \le i \le n$, the number of leftmost $i$-boxes in row $n-i+1$ is strictly greater than the total number of boxes in row $n-i+2$.
\end{enumerate} 
\end{dfn}

\begin{ex}
Consider the following semistandard decomposition tableaux of shape $(4,3,1)$ for $\q(3)$:
\[
\text{dual large: }\quad\gyoung(3333,:;222,::;1)\,,
\qquad\qquad
\text{not dual large: }\quad\gyoung(3333,:;212,::;1)\,.
\]
\end{ex}

\begin{dfn}
A semistandard decomposition tableau $T$ for $\q(n)$ is called \defn{dual marginally large} if
\begin{enumerate}
\item it has $n$ rows, and
\item for all $1 \le i \le n$, the number of leftmost $i$-boxes in row $n-i+1$ is greater than the total number of boxes in row $n-i+2$ by exactly one.
\end{enumerate}
Denote the set of all dual marginally large semistandard tableaux for $\q(n)$ by $\SDT(-\infty)$.
\end{dfn}

We note that a dual marginally large tableau is dual large.

\begin{ex}\label{ex:lgvsmg}
Consider the following semistandard decomposition tableaux for $\q(3)$:
\[
\text{not dual marginally large: }\gyoung(3333,:;222,::;1)\,,
\qquad
\text{dual marginally large: }\gyoung(333,:;22,::;1)\,.
\]
\end{ex}

A column $C$ of height $h$ is called \defn{trivial} if the column is $C = [n, n-1, \dotsc, n+1-h]^\top$ and every entry to the left of an $i \in C$ is $i$ (if it exists).
We will say we \defn{push in} (resp.\ \defn{push out}) a trivial column of height $h$ from $C$ if we add (resp.\ remove) a box with an $i$ to the left of row $n+1-i$ and sliding the remaining boxes back to a shifted shape.
Furthermore, in the case that a tableau is dual large, the act of pushing in (resp.\ out) a column of height $r$ is equivalent to adding (resp.\ removing) a trivial column of height $r$. However, we will need to consider pushing in/out columns for the case when the tableau is not dual large.

\begin{ex}
In the following, a trivial column of height $2$ is pushed in to the given semistandard decomposition tableaux, where the pushed in trivial column are the shaded boxes:
\[
\gyoung(3333,:;222,::;1)
\longrightarrow
\gyoung(!<\fb>3!<\fe>3333,:!<\fb>2!<\fe>222,:::;1)
\longrightarrow
\gyoung(!<\fb>3!<\fe>3333,:!<\fb>2!<\fe>222,::;1)\,. 
\]
\end{ex}

The following lemma is straightforward from the definitions, where the uniqueness follows from the diamond lemma.

\begin{lemma}
\label{lemma:unique_marginally_large}
Given any dual large tableau $T$, we can construct a unique dual marginally large tableau $T_M$ by pushing out to the left trivial columns.
\end{lemma}

\begin{ex}
In Example~\ref{ex:lgvsmg}, the tableau on the left has one trivial column: the column $[3,2]^\top$ on the right.  Removing this column yields the dual marginally large tableau on the right in the same example.
\end{ex}

\begin{dfn}\label{dfn:up_operators}
Let $T \in \SDT(-\infty)$ for $\q(n)$.
\begin{enumerate}
\item Suppose $i \in I_0$.
\begin{enumerate}
\item\label{item:upe} Let $T'$ be the tableau obtained from $T$ by changing the $(i+1)$-box corresponding to the leftmost $i+1 \in \red_i(T)$ to an $i$-box.  If $T'$ is dual marginally large, then $T' = e_iT$.  Otherwise, let $T''$ be the tableau obtained from $T'$ by adding a $(n-k+1)$-box in row $k$, for each $1\le k \le n-i+1$.  Then $T'' = e_iT$.
\item\label{item:upf} If there is no such $i \in \red_i(T)$, then $f_iT = \zero$. Otherwise, let $T'$ be the tableau obtained from $T$ by changing the $i$-box corresponding to the rightmost $i \in \red_i(T)$ to an $(i+1)$-box.  If $T'$ is dual marginally large, then $T' = f_iT$.  Otherwise, let $T''$ be the tableau obtained from $T'$ by removing a $(n-k+1)$-box in row $k$, for each $1\le k \le n-i+1$.  Then $T'' = f_iT$.
\end{enumerate}
\item Both $e_{\bon}$ and $f_{\bon}$ are defined exactly as in Definition \ref{dfn:finite_operators}, except for the need to maintain the dual marginally large condition as in (\ref{item:upe}) and (\ref{item:upf}) above.
\end{enumerate}
\end{dfn}

\begin{ex}
Let $n=3$ and
\[
T = \gyoung(33332,:;221,::;1) \in \SDT(-\infty).
\]
Then $\rd(T) = 233331221$.  After pairing off all possible $(2,1)$, there is no $1$ remaining and the leftmost $2$ remaining corresponds to the underlined $\underline{\color{red}2}$ in $233331\underline{\color{red}2}21$.  Hence $f_1T = \zero$, but
\[
e_1T = \gyoung(333332,:;2211,::;1)
\qquad\text{and}\qquad
e_{\overline{1}}T = \gyoung(33331,:;221,::;1).
\]
Note that a $3$-box needed to be added to the first row and a $2$-box needed to be added to the second row to maintain the dual marginally large condition in $e_1T$. In other words, we pushed in a trivial column of height $2$ (from the left).

After pairing off all possible $(3,2)$ in $\rd(T)$, the leftmost $3$ and rightmost $2$ remaining are the underlined letters in $\underline{\color{red}2}\underline{\color{red}3}3331221$.  Hence
\[
e_2T = \gyoung(333322,:;221,::;1)
\qquad\text{and}\qquad
f_2T = \gyoung(3333,:;221,::;1)\ .
\]
Note that a $3$-box needed to be added to the first row of $e_2T$ and removed from the first row of $f_2T$ to maintain the dual marginally large condition. Note that a $3$-box is a trivial column of height $1$.

A diagram of the crystal graph $\SDT(-\infty)$ up to height $3$ is included in Figure~\ref{fig:up_graph}.
\end{ex}

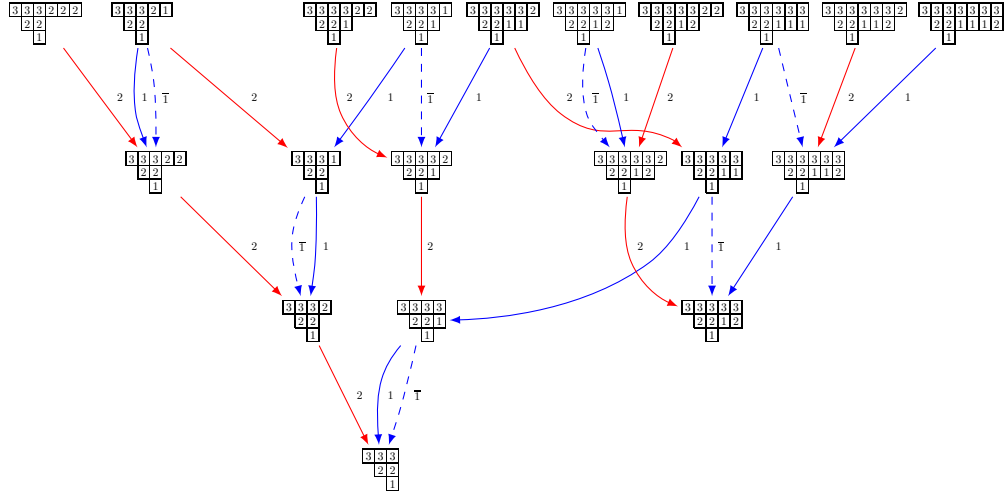
\begin{figure}[t]
\[
\begin{tikzpicture}[>=latex,line join=bevel,,every node/.style={scale=.4},yscale=.55,xscale=0.375]
  \node (node_12) at (757.0bp,328.0bp) [draw,draw=none] {${\def\lr#1{\multicolumn{1}{|@{\hspace{.6ex}}c@{\hspace{.6ex}}|}{\raisebox{-.3ex}{$#1$}}}\raisebox{-.6ex}{$\begin{array}[b]{*{6}c}\cline{1-6}\lr{3}&\lr{3}&\lr{3}&\lr{3}&\lr{3}&\lr{1}\\\cline{1-6}&\lr{2}&\lr{2}&\lr{1}&\lr{2}\\\cline{2-5}&&\lr{1}\\\cline{3-3}\end{array}$}}$};
  \node (node_26) at (473.0bp,124.0bp) [draw,draw=none] {${\def\lr#1{\multicolumn{1}{|@{\hspace{.6ex}}c@{\hspace{.6ex}}|}{\raisebox{-.3ex}{$#1$}}}\raisebox{-.6ex}{$\begin{array}[b]{*{4}c}\cline{1-4}\lr{3}&\lr{3}&\lr{3}&\lr{2}\\\cline{1-4}&\lr{2}&\lr{2}\\\cline{2-3}&&\lr{1}\\\cline{3-3}\end{array}$}}$};
  \node (node_27) at (588.0bp,226.0bp) [draw,draw=none] {${\def\lr#1{\multicolumn{1}{|@{\hspace{.6ex}}c@{\hspace{.6ex}}|}{\raisebox{-.3ex}{$#1$}}}\raisebox{-.6ex}{$\begin{array}[b]{*{5}c}\cline{1-5}\lr{3}&\lr{3}&\lr{3}&\lr{3}&\lr{2}\\\cline{1-5}&\lr{2}&\lr{2}&\lr{1}\\\cline{2-4}&&\lr{1}\\\cline{3-3}\end{array}$}}$};
  \node (node_24) at (849.0bp,328.0bp) [draw,draw=none] {${\def\lr#1{\multicolumn{1}{|@{\hspace{.6ex}}c@{\hspace{.6ex}}|}{\raisebox{-.3ex}{$#1$}}}\raisebox{-.6ex}{$\begin{array}[b]{*{7}c}\cline{1-7}\lr{3}&\lr{3}&\lr{3}&\lr{3}&\lr{3}&\lr{2}&\lr{2}\\\cline{1-7}&\lr{2}&\lr{2}&\lr{1}&\lr{2}\\\cline{2-5}&&\lr{1}\\\cline{3-3}\end{array}$}}$};
  \node (node_22) at (506.0bp,328.0bp) [draw,draw=none] {${\def\lr#1{\multicolumn{1}{|@{\hspace{.6ex}}c@{\hspace{.6ex}}|}{\raisebox{-.3ex}{$#1$}}}\raisebox{-.6ex}{$\begin{array}[b]{*{6}c}\cline{1-6}\lr{3}&\lr{3}&\lr{3}&\lr{3}&\lr{2}&\lr{2}\\\cline{1-6}&\lr{2}&\lr{2}&\lr{1}\\\cline{2-4}&&\lr{1}\\\cline{3-3}\end{array}$}}$};
  \node (node_21) at (321.0bp,226.0bp) [draw,draw=none] {${\def\lr#1{\multicolumn{1}{|@{\hspace{.6ex}}c@{\hspace{.6ex}}|}{\raisebox{-.3ex}{$#1$}}}\raisebox{-.6ex}{$\begin{array}[b]{*{5}c}\cline{1-5}\lr{3}&\lr{3}&\lr{3}&\lr{2}&\lr{2}\\\cline{1-5}&\lr{2}&\lr{2}\\\cline{2-3}&&\lr{1}\\\cline{3-3}\end{array}$}}$};
  \node (node_10) at (588.0bp,328.0bp) [draw,draw=none] {${\def\lr#1{\multicolumn{1}{|@{\hspace{.6ex}}c@{\hspace{.6ex}}|}{\raisebox{-.3ex}{$#1$}}}\raisebox{-.6ex}{$\begin{array}[b]{*{5}c}\cline{1-5}\lr{3}&\lr{3}&\lr{3}&\lr{3}&\lr{1}\\\cline{1-5}&\lr{2}&\lr{2}&\lr{1}\\\cline{2-4}&&\lr{1}\\\cline{3-3}\end{array}$}}$};
  \node (node_28) at (670.0bp,328.0bp) [draw,draw=none] {${\def\lr#1{\multicolumn{1}{|@{\hspace{.6ex}}c@{\hspace{.6ex}}|}{\raisebox{-.3ex}{$#1$}}}\raisebox{-.6ex}{$\begin{array}[b]{*{6}c}\cline{1-6}\lr{3}&\lr{3}&\lr{3}&\lr{3}&\lr{3}&\lr{2}\\\cline{1-6}&\lr{2}&\lr{2}&\lr{1}&\lr{1}\\\cline{2-5}&&\lr{1}\\\cline{3-3}\end{array}$}}$};
  \node (node_29) at (798.0bp,226.0bp) [draw,draw=none] {${\def\lr#1{\multicolumn{1}{|@{\hspace{.6ex}}c@{\hspace{.6ex}}|}{\raisebox{-.3ex}{$#1$}}}\raisebox{-.6ex}{$\begin{array}[b]{*{6}c}\cline{1-6}\lr{3}&\lr{3}&\lr{3}&\lr{3}&\lr{3}&\lr{2}\\\cline{1-6}&\lr{2}&\lr{2}&\lr{1}&\lr{2}\\\cline{2-5}&&\lr{1}\\\cline{3-3}\end{array}$}}$};
  \node (node_9) at (482.0bp,226.0bp) [draw,draw=none] {${\def\lr#1{\multicolumn{1}{|@{\hspace{.6ex}}c@{\hspace{.6ex}}|}{\raisebox{-.3ex}{$#1$}}}\raisebox{-.6ex}{$\begin{array}[b]{*{4}c}\cline{1-4}\lr{3}&\lr{3}&\lr{3}&\lr{1}\\\cline{1-4}&\lr{2}&\lr{2}\\\cline{2-3}&&\lr{1}\\\cline{3-3}\end{array}$}}$};
  \node (node_6) at (307.0bp,328.0bp) [draw,draw=none] {${\def\lr#1{\multicolumn{1}{|@{\hspace{.6ex}}c@{\hspace{.6ex}}|}{\raisebox{-.3ex}{$#1$}}}\raisebox{-.6ex}{$\begin{array}[b]{*{5}c}\cline{1-5}\lr{3}&\lr{3}&\lr{3}&\lr{2}&\lr{1}\\\cline{1-5}&\lr{2}&\lr{2}\\\cline{2-3}&&\lr{1}\\\cline{3-3}\end{array}$}}$};
  \node (node_40) at (1130.0bp,328.0bp) [draw,draw=none] {${\def\lr#1{\multicolumn{1}{|@{\hspace{.6ex}}c@{\hspace{.6ex}}|}{\raisebox{-.3ex}{$#1$}}}\raisebox{-.6ex}{$\begin{array}[b]{*{7}c}\cline{1-7}\lr{3}&\lr{3}&\lr{3}&\lr{3}&\lr{3}&\lr{3}&\lr{3}\\\cline{1-7}&\lr{2}&\lr{2}&\lr{1}&\lr{1}&\lr{1}&\lr{2}\\\cline{2-7}&&\lr{1}\\\cline{3-3}\end{array}$}}$};
  \node (node_38) at (977.0bp,226.0bp) [draw,draw=none] {${\def\lr#1{\multicolumn{1}{|@{\hspace{.6ex}}c@{\hspace{.6ex}}|}{\raisebox{-.3ex}{$#1$}}}\raisebox{-.6ex}{$\begin{array}[b]{*{6}c}\cline{1-6}\lr{3}&\lr{3}&\lr{3}&\lr{3}&\lr{3}&\lr{3}\\\cline{1-6}&\lr{2}&\lr{2}&\lr{1}&\lr{1}&\lr{2}\\\cline{2-6}&&\lr{1}\\\cline{3-3}\end{array}$}}$};
  \node (node_18) at (210.0bp,328.0bp) [draw,draw=none] {${\def\lr#1{\multicolumn{1}{|@{\hspace{.6ex}}c@{\hspace{.6ex}}|}{\raisebox{-.3ex}{$#1$}}}\raisebox{-.6ex}{$\begin{array}[b]{*{6}c}\cline{1-6}\lr{3}&\lr{3}&\lr{3}&\lr{2}&\lr{2}&\lr{2}\\\cline{1-6}&\lr{2}&\lr{2}\\\cline{2-3}&&\lr{1}\\\cline{3-3}\end{array}$}}$};
  \node (node_33) at (547.0bp,22.0bp) [draw,draw=none] {${\def\lr#1{\multicolumn{1}{|@{\hspace{.6ex}}c@{\hspace{.6ex}}|}{\raisebox{-.3ex}{$#1$}}}\raisebox{-.6ex}{$\begin{array}[b]{*{3}c}\cline{1-3}\lr{3}&\lr{3}&\lr{3}\\\cline{1-3}&\lr{2}&\lr{2}\\\cline{2-3}&&\lr{1}\\\cline{3-3}\end{array}$}}$};
  \node (node_35) at (880.0bp,226.0bp) [draw,draw=none] {${\def\lr#1{\multicolumn{1}{|@{\hspace{.6ex}}c@{\hspace{.6ex}}|}{\raisebox{-.3ex}{$#1$}}}\raisebox{-.6ex}{$\begin{array}[b]{*{5}c}\cline{1-5}\lr{3}&\lr{3}&\lr{3}&\lr{3}&\lr{3}\\\cline{1-5}&\lr{2}&\lr{2}&\lr{1}&\lr{1}\\\cline{2-5}&&\lr{1}\\\cline{3-3}\end{array}$}}$};
  \node (node_34) at (588.0bp,124.0bp) [draw,draw=none] {${\def\lr#1{\multicolumn{1}{|@{\hspace{.6ex}}c@{\hspace{.6ex}}|}{\raisebox{-.3ex}{$#1$}}}\raisebox{-.6ex}{$\begin{array}[b]{*{4}c}\cline{1-4}\lr{3}&\lr{3}&\lr{3}&\lr{3}\\\cline{1-4}&\lr{2}&\lr{2}&\lr{1}\\\cline{2-4}&&\lr{1}\\\cline{3-3}\end{array}$}}$};
  \node (node_37) at (941.0bp,328.0bp) [draw,draw=none] {${\def\lr#1{\multicolumn{1}{|@{\hspace{.6ex}}c@{\hspace{.6ex}}|}{\raisebox{-.3ex}{$#1$}}}\raisebox{-.6ex}{$\begin{array}[b]{*{6}c}\cline{1-6}\lr{3}&\lr{3}&\lr{3}&\lr{3}&\lr{3}&\lr{3}\\\cline{1-6}&\lr{2}&\lr{2}&\lr{1}&\lr{1}&\lr{1}\\\cline{2-6}&&\lr{1}\\\cline{3-3}\end{array}$}}$};
  \node (node_36) at (880.0bp,124.0bp) [draw,draw=none] {${\def\lr#1{\multicolumn{1}{|@{\hspace{.6ex}}c@{\hspace{.6ex}}|}{\raisebox{-.3ex}{$#1$}}}\raisebox{-.6ex}{$\begin{array}[b]{*{5}c}\cline{1-5}\lr{3}&\lr{3}&\lr{3}&\lr{3}&\lr{3}\\\cline{1-5}&\lr{2}&\lr{2}&\lr{1}&\lr{2}\\\cline{2-5}&&\lr{1}\\\cline{3-3}\end{array}$}}$};
  \node (node_31) at (1033.0bp,328.0bp) [draw,draw=none] {${\def\lr#1{\multicolumn{1}{|@{\hspace{.6ex}}c@{\hspace{.6ex}}|}{\raisebox{-.3ex}{$#1$}}}\raisebox{-.6ex}{$\begin{array}[b]{*{7}c}\cline{1-7}\lr{3}&\lr{3}&\lr{3}&\lr{3}&\lr{3}&\lr{3}&\lr{2}\\\cline{1-7}&\lr{2}&\lr{2}&\lr{1}&\lr{1}&\lr{2}\\\cline{2-6}&&\lr{1}\\\cline{3-3}\end{array}$}}$};
  \definecolor{strokecol}{rgb}{0.0,0.0,0.0};
  \pgfsetstrokecolor{strokecol}
  \draw [blue,->] (node_40) ..controls (1074.7bp,290.88bp) and (1042.4bp,269.75bp)  .. (node_38);
  \draw (1077.0bp,277.0bp) node {$1$};
  \draw [blue,->] (node_10) ..controls (554.97bp,294.58bp) and (539.22bp,279.33bp)  .. (525.0bp,266.0bp) .. controls (520.96bp,262.21bp) and (516.66bp,258.25bp)  .. (node_9);
  \draw (557.0bp,277.0bp) node {$1$};
  \draw [red,->] (node_22) ..controls (499.43bp,293.63bp) and (499.46bp,277.78bp)  .. (507.0bp,266.0bp) .. controls (516.47bp,251.19bp) and (533.29bp,241.87bp)  .. (node_27);
  \draw (516.0bp,277.0bp) node {$2$};
  \draw [red,->] (node_18) ..controls (249.67bp,291.26bp) and (272.38bp,270.8bp)  .. (node_21);
  \draw (285.0bp,277.0bp) node {$2$};
  \draw [red,->] (node_28) ..controls (694.47bp,292.63bp) and (710.02bp,275.46bp)  .. (728.0bp,266.0bp) .. controls (773.01bp,242.32bp) and (792.82bp,264.28bp)  .. (841.0bp,248.0bp) .. controls (841.1bp,247.97bp) and (841.2bp,247.93bp)  .. (node_35);
  \draw (737.0bp,277.0bp) node {$2$};
  \draw [blue,->] (node_35) ..controls (852.83bp,190.73bp) and (835.88bp,173.56bp)  .. (817.0bp,164.0bp) .. controls (753.59bp,131.87bp) and (668.93bp,125.5bp)  .. (node_34);
  \draw (855.0bp,175.0bp) node {$1$};
  \draw [red,->] (node_6) ..controls (367.38bp,292.5bp) and (416.4bp,264.49bp)  .. (node_9);
  \draw (420.0bp,277.0bp) node {$2$};
  \draw [dashed,blue,->] (node_37) ..controls (953.53bp,292.2bp) and (960.32bp,273.34bp)  .. (node_38);
  \draw (972.0bp,277.0bp) node {$\overline{1}$};
  \draw [blue,->] (node_6) ..controls (299.3bp,294.49bp) and (297.8bp,279.24bp)  .. (301.0bp,266.0bp) .. controls (301.73bp,262.99bp) and (302.73bp,259.95bp)  .. (node_21);
  \draw (310.0bp,277.0bp) node {$1$};
  \draw [dashed,blue,->] (node_6) ..controls (316.31bp,300.38bp) and (318.0bp,294.01bp)  .. (319.0bp,288.0bp) .. controls (320.63bp,278.25bp) and (321.34bp,267.51bp)  .. (node_21);
  \draw (331.0bp,277.0bp) node {$\overline{1}$};
  \draw [blue,->] (node_37) ..controls (919.61bp,291.93bp) and (907.83bp,272.62bp)  .. (node_35);
  \draw (925.0bp,277.0bp) node {$1$};
  \draw [red,->] (node_24) ..controls (831.12bp,291.93bp) and (821.27bp,272.62bp)  .. (node_29);
  \draw (838.0bp,277.0bp) node {$2$};
  \draw [blue,->] (node_9) ..controls (482.73bp,192.49bp) and (482.47bp,177.37bp)  .. (481.0bp,164.0bp) .. controls (480.7bp,161.26bp) and (480.3bp,158.43bp)  .. (node_26);
  \draw (492.0bp,175.0bp) node {$1$};
  \draw [dashed,blue,->] (node_9) ..controls (463.81bp,198.76bp) and (460.74bp,192.37bp)  .. (459.0bp,186.0bp) .. controls (456.3bp,176.12bp) and (457.49bp,165.25bp)  .. (node_26);
  \draw (468.0bp,175.0bp) node {$\overline{1}$};
  \draw [red,->] (node_21) ..controls (376.04bp,188.79bp) and (413.43bp,164.19bp)  .. (node_26);
  \draw (420.0bp,175.0bp) node {$2$};
  \draw [red,->] (node_26) ..controls (499.15bp,87.665bp) and (513.78bp,67.897bp)  .. (node_33);
  \draw (526.0bp,73.0bp) node {$2$};
  \draw [blue,->] (node_38) ..controls (942.46bp,189.4bp) and (922.84bp,169.17bp)  .. (node_36);
  \draw (947.0bp,175.0bp) node {$1$};
  \draw [dashed,blue,->] (node_10) ..controls (588.0bp,292.34bp) and (588.0bp,273.7bp)  .. (node_27);
  \draw (597.0bp,277.0bp) node {$\overline{1}$};
  \draw [blue,->] (node_12) ..controls (770.51bp,300.33bp) and (773.48bp,293.97bp)  .. (776.0bp,288.0bp) .. controls (780.21bp,278.04bp) and (784.35bp,266.98bp)  .. (node_29);
  \draw (794.0bp,277.0bp) node {$1$};
  \draw [dashed,blue,->] (node_12) ..controls (749.21bp,294.1bp) and (748.34bp,278.57bp)  .. (754.0bp,266.0bp) .. controls (755.74bp,262.13bp) and (758.08bp,258.48bp)  .. (node_29);
  \draw (763.0bp,277.0bp) node {$\overline{1}$};
  \draw [blue,->] (node_34) ..controls (557.43bp,99.215bp) and (551.46bp,92.015bp)  .. (548.0bp,84.0bp) .. controls (543.95bp,74.616bp) and (542.81bp,63.641bp)  .. (node_33);
  \draw (557.0bp,73.0bp) node {$1$};
  \draw [dashed,blue,->] (node_34) ..controls (577.11bp,90.524bp) and (571.61bp,75.274bp)  .. (566.0bp,62.0bp) .. controls (564.78bp,59.11bp) and (563.45bp,56.128bp)  .. (node_33);
  \draw (584.0bp,73.0bp) node {$\overline{1}$};
  \draw [dashed,blue,->] (node_35) ..controls (880.0bp,190.34bp) and (880.0bp,171.7bp)  .. (node_36);
  \draw (889.0bp,175.0bp) node {$\overline{1}$};
  \draw [red,->] (node_29) ..controls (791.43bp,191.63bp) and (791.46bp,175.78bp)  .. (799.0bp,164.0bp) .. controls (808.47bp,149.19bp) and (825.29bp,139.87bp)  .. (node_36);
  \draw (808.0bp,175.0bp) node {$2$};
  \draw [red,->] (node_31) ..controls (1013.4bp,291.93bp) and (1002.5bp,272.62bp)  .. (node_38);
  \draw (1020.0bp,277.0bp) node {$2$};
  \draw [blue,->] (node_28) ..controls (641.03bp,291.67bp) and (624.82bp,271.9bp)  .. (node_27);
  \draw (646.0bp,277.0bp) node {$1$};
  \draw [red,->] (node_27) ..controls (588.0bp,190.34bp) and (588.0bp,171.7bp)  .. (node_34);
  \draw (597.0bp,175.0bp) node {$2$};
\end{tikzpicture}
\]
\caption{A bottom portion of the $\q(3)$-crystal $\SDT(-\infty)$ containing the lowest weight element $L^{-\infty}$ created using \textsc{SageMath}~\cite{sage}.}\label{fig:up_graph}
\end{figure}

Now we show that the standard decomposition tableaux form a directed system.

\begin{lemma}\label{lem:up_inclusions}
Suppose $\lambda,\mu \in \Lambda^-$.  Then there exists a $\q(n)$-crystal embedding 
\[
\upsilon_{\lambda,\lambda+\mu} \colon \SDT(\lambda)\otimes\T_{-\lambda} \longhookarrow \SDT(\lambda+\mu) \otimes \T_{-\lambda-\mu}
\] 
such that $L^{\lambda} \otimes t_{-\lambda} \mapsto L^{\lambda+\mu}\otimes t_{-\lambda-\mu}$. 
\end{lemma}

\begin{proof}
For $T \in \SDT(\lambda)$, define $E(T)$ to be the tableau obtained from $T$ by adding $\mu_{n+1-i}$ $i$-boxes on the left to row $n+1-i$, for each $1\le i \le n$. It is evident that the result has shape $w_0(\lambda + \mu)$.  Define
\[
\upsilon_{\lambda,\lambda+\mu} \colon \SDT(\lambda)\otimes\T_{-\lambda} \longhookarrow \SDT(\lambda+\mu) \otimes \T_{-\lambda-\mu}
\] 
by $(T \otimes t_{-\lambda}) \mapsto E(T) \otimes t_{-\lambda-\mu}$.  Immediately, we have $\upsilon_{\lambda,\lambda+\mu}(L^{\lambda} \otimes t_{-\lambda}) = L^{\lambda+\mu} \otimes t_{-\lambda-\mu}$ and
\[
\varepsilon_i(L^{\lambda} \otimes t_{-\lambda}) = \varepsilon_i(L^{\lambda+\mu} \otimes t_{-\lambda-\mu}),
\qquad\qquad
\varphi_i(L^{\lambda} \otimes t_{-\lambda}) = \varphi_i(L^{\lambda+\mu} \otimes t_{-\lambda-\mu}),
\]
(which implies $\wt_i(L^{\lambda} \otimes t_{-\lambda}) = \wt_i(L^{\lambda+\mu} \otimes t_{-\lambda-\mu})$, for $i\in I_0$, by the crystal axiom~(\ref{axiom:phi_ep_wt})). Since the crystal $\SDT(\nu) \otimes t_{-\nu}$ is generated by $L^{\nu} \otimes t_{-\nu}$, the crystal axioms imply it suffices to show that $E(x_i T) = x_i E(T)$ for all $x_i = e_i,f_i$ and $T \in \SDT(\lambda)$ such that $x_i T \neq 0$.

Fix some $T \in \SDT(\lambda)$ and $i \in I_0$. From the definition of the crystal operators, we need to show that the difference between $\red_i(T)$ and $\red_i\bigl(E(T)\bigr)$ is possibly some additional $i$'s on the left (resp.~$(i+1)$'s on the right) corresponding to one of the added $i$-boxes (resp.~$(i+1)$-boxes) in $E(T)$. By induction and the construction, it is sufficient to consider $w_0\mu$ being a column of height $h$. The case when $h < n-i$ is trivial as no $i$ nor $i+1$ is added to $T$.

We note that Proposition~\ref{prop:decomp_characterization}(\ref{char:leftmost_box}) implies that every entry in row $n+1-i$ cannot have value larger than $i$ as otherwise the first row would have a value of $n+1$ or larger. Hence the added $i+1$ (to row $n-i$) will either be the rightmost uncanceled $i+1$ or will cancel with some $i$. If the added $i+1$ in row $n-i$ cancels with an $i$ (which must be in row $n-i+1$) in $E(T)$, then there must exist an $i+1$ in the first position of row $n-i$ in $T$ as otherwise this would violate Proposition~\ref{prop:decomp_characterization}(\ref{char:leftmost_box}). If there is no other $i$ in row $n-i+1$, then the added $i$ in $E(T)$ would then cancel with the $i+1$ from $T$. Now suppose there exists a second $i$ in row $n-i+1$ in $T$, and without loss of generality, assume it is the leftmost $i$ box in row $n-i+1$ of $T$. Then there must exist another $(i+1)$-box $b$ in row $n-i$ of $T$ as otherwise we obtain a type L configuration. There cannot be an $i$-box to the left of $b$ in row $n-i$ of $T$ as otherwise this would violate the hook word condition as the entry above the rightmost $i$ in row $n-i+1$ must be at most $i+1$ by Proposition~\ref{prop:decomp_characterization}(\ref{char:leftmost_box}) as above.
Repeating this argument, we see that (locally around the rows $n-i$ and $n-i+1$)
\begin{align*}
\rd_i(T) & = w' \underbrace{i+1, \dotsc, i+1}_{k} \underbrace{i, \dotsc, i}_{k}\,,
\\
\rd_i\bigl(E(T)\bigr) & = w' \underbrace{i+1, \dotsc, i+1, {\color{red} \underline{i+1}}}_{k+1} \underbrace{i, \dotsc, i, {\color{red} \underline{i}}}_{k+1}\,,
\end{align*}
where the underlined entries are those added to form $E(T)$ and $w'$ is some word unaffected when constructing $E(T)$. Thus we cannot obtain an additional $i$ on the right of $\rd_i\bigl(E(T)\bigr)$ coming from the $i$-box added to row $n-i+1$.

To show $E(f_{\bon}T) = f_{\bon} E(T)$, note that the only problem that can arise is when pushing in a trivial column that contains a $2$ to some $T \in \SDT(\lambda)$ such that the only $1$ is in row $n$ but no $2$ in row $n-1$ of $T$ (as this would make $f_{\bon} E(T) = \zero$). However, all entries in row $n-1$ must be at most $2$ as a consequence of Proposition~\ref{prop:decomp_characterization}(\ref{char:leftmost_box}) as shown above. Thus there must be a $2$ in row $n-1$, which shows we must have $E(f_{\bon}T) = f_{\bon} E(T)$.
For $E(e_{\bon}T) = e_{\bon} E(T)$, note that the only way this could not happen is if there exists a $2$ in row $n$ and we push in a trivial column containing a $2$ (as this would come before the $2$ we act upon for $e_{\bon} T$). However, this is impossible as it follows from Proposition~\ref{prop:decomp_characterization}(\ref{char:leftmost_box}) that there cannot exist a $2$ in row $n$.
\end{proof}

\begin{cor}\label{cor:up_system}
The collection $\{\SDT(\lambda)\otimes\T_{-\lambda}\}_{\lambda\in\Lambda^-}$ together with the inclusion maps from Lemma \ref{lem:up_inclusions} form a directed system.
\end{cor}

To prove the corollary, one makes repeated use of Lemma~\ref{lem:up_inclusions} applied to diagrams of the following form:
\[
\label{eq:directed_system}
\begin{tikzpicture}[xscale=7,yscale=2,baseline=0.6cm]
\node (1) at (0,1) {$\SDT(\lambda)\otimes\T_{-\lambda}$};
\node (2) at (1,1) {$\SDT(\lambda+\mu) \otimes \T_{-\lambda-\mu}$};
\node (3) at (1,0) {$\SDT(\lambda+\mu+\xi) \otimes \T_{-\lambda-\mu-\xi}$.};
\path[->,font=\scriptsize]
 (1) edge node[above]{$\upsilon_{\lambda+\mu,\lambda}$} (2)
 (1) edge node[below left]{$\upsilon_{\lambda+\mu+\xi,\lambda}$} (3)
 (2) edge node[right]{$\upsilon_{\lambda+\mu+\xi,\lambda+\mu}$} (3);
\end{tikzpicture}
\]

Define $L^{-\infty}$ to be the decomposition tableau $L^{-n\epsilon_1 - (n-1)\epsilon_2 -\cdots-\epsilon_n}$.

\begin{thm}\label{thm:up-inf}
The set $\SDT(-\infty)$ together with $e_i,f_i$ from Definition \ref{dfn:up_operators} is an abstract $\q(n)$-crystal such that
\[
\SDT(-\infty) \cong \varinjlim_{\lambda\in\Lambda^-} \SDT(\lambda) \otimes \T_{-\lambda}.
\]
\end{thm}

\begin{proof}
Let $\overrightarrow{\SDT}$ denote the direct limit. We claim the map $\psi \colon \SDT(-\infty) \longrightarrow \overrightarrow{\SDT}$ given by $\psi(T) = [T]$ is the desired crystal isomorphism.

For any $T \in \SDT(-\infty)$ of shape $w_0\lambda$, we note there exists a projection
\[
\pi_T \colon \SDT(-\infty) \longrightarrow \SDT(\lambda) \otimes \T_{-\lambda}
\]
such that for any $T' = f_{i_1} \cdots f_{i_{\ell}} T$, we can form $\pi_T(T')$ by adding in suitably many trivial columns to $T'$ until we obtain the shape $w_0\lambda$ (and adding the inconsequential tensor factor $t_{-\lambda}$). The latter follows by induction since $\pi_T(f_i T') = f_i \pi_T(T')$, where this equality was shown in the proof of Lemma~\ref{lem:up_inclusions} as $f_i$ commutes with the procedure of adding or removing trivial columns. Therefore, we have $\psi(f_iT) = [f_i T] = f_i [T] = f_i \psi(T)$. 

Next, it was shown in the proof of Lemma~\ref{lem:up_inclusions} that any two $\tau, \tau' \in [T]$ differ only by trivial columns. Hence, we have $\psi(e_i T) = e_i \psi(T)$.  It is clear that $\psi$ satisfies the rest of the properties of being a crystal morphism. Since there is a unique dual marginally large seminstandard decomposition tableau in each class $[T] \in \overrightarrow{\SDT}$, we have that $\psi$ is a bijection.
\end{proof}

\subsection{Comparison of characters}
Let us provide some evidence that $\SDT(-\infty)$ is a combinatorial model of $B(-\infty)$. We show that the character of $\SDT(-\infty)$ agrees with that of the upper half of the quantum group $U_q^+\bigl(\q(n)\bigr)$.

\begin{prop}\label{prop:char}
We have
\[
\ch \SDT(-\infty) = \ch U_q^+\bigl(\q(n)\bigr).
\]
\end{prop}

\begin{proof}
From~\cite[Sec.~2.3.1]{CW12}, it is straightforward to see that
\begin{equation}
\label{eq:Verma_character}
\ch U_q^+\bigl(\q(n)\bigr) = \ch U_q^+\bigl(\gl(n)\bigr) \prod_{\alpha \in \Phi^+} (1 + e^{\alpha})
= \dfrac{\displaystyle\prod_{\alpha \in \Phi^+} (1 + e^{\alpha})}{\displaystyle\prod_{\alpha \in \Phi^+} (1 - e^{\alpha})}, 
\end{equation}
where $\Phi^+$ is set of positive roots associated to $\gl(n)$ (see also~\cite[Sec.~3]{Cheng17}).\footnote{Note that the character of a Verma module of weight $0$ after substituting $\epsilon_i \mapsto -\epsilon_i$ is equal to $\ch U_q^+\bigl(\q(n)\bigr)$.}
Next, for $\lambda \in \Lambda^-$ from~\cite{CK16,PS97II} we have
\[
\ch \SDT(\lambda) = \dfrac{\displaystyle\prod_{\alpha \in \Phi^+} (1 + e^{-\alpha})}{\displaystyle\prod_{\alpha \in \Phi^+} (1 - e^{-\alpha})} \sum_{w \in \mathfrak{S}_n} (-1)^{\ell(w)} w \left( \dfrac{e^{w_0 \lambda}}{\displaystyle\prod_{\alpha \in \Phi^+(w_0 \lambda)} (1 + e^{-\alpha})} \right),
\]
where $\Phi^+(\mu) := \{ \epsilon_i - \epsilon_j \mid \mu_i = \mu_j \ (i < j) \}$ and $\mathfrak{S}_n$ is the symmetric group on $n$ letters. 
For the direct limit, it is sufficient to take $\lambda_1 > \lambda_2 > \cdots > \lambda_n > 0$, and thus we have $\Phi^+(w_0 \lambda) = \emptyset$. It follows that
\begin{align*}
\ch \bigl( \SDT(\lambda) \otimes \T_{-\lambda} \bigr) & = \dfrac{\displaystyle\prod_{\alpha \in \Phi^+} (1 + e^{-\alpha})}{\displaystyle\prod_{\alpha \in \Phi^+} (1 - e^{-\alpha})} e^{-\lambda} \sum_{w \in \mathfrak{S}_n} (-1)^{\ell(w)} w e^{w_0 \lambda}
\\ & = \dfrac{\displaystyle\prod_{\alpha \in \Phi^+} (e^{\alpha} + 1)}{\displaystyle\prod_{\alpha \in \Phi^+} (e^{\alpha} - 1)} \sum_{w' \in \mathfrak{S}_n} (-1)^{\ell(w'w_0)} e^{w' \lambda - \lambda},
\end{align*}
where we rewrite the sum by $w' = w w_0$. Next, we recall that $\absval{\Phi^+} = \ell(w_0)$ since the length of a permutation equals the number of inversions (which are positive roots sent to a negative root), we have
\[
\ch \bigl( \SDT(\lambda) \otimes \T_{-\lambda} \bigr) = \dfrac{\displaystyle\prod_{\alpha \in \Phi^+} (1 + e^{\alpha})}{\displaystyle\prod_{\alpha \in \Phi^+} (1 - e^{\alpha})} \sum_{w \in \mathfrak{S}_n} (-1)^{\ell(w)} e^{w \lambda - \lambda}.
\]
Define $\rho = \sum_{i \in I_0} \Lambda_i$.
Now for any $\mu \in Q^+$, we note that there exists a $k$ such that for all $\lambda > -k \rho$ the coefficient of $e^{\mu}$ is $0$ in
\[
\dfrac{\displaystyle\prod_{\alpha \in \Phi^+} (1 + e^{\alpha})}{\displaystyle\prod_{\alpha \in \Phi^+} (1 - e^{\alpha})} \sum_{w \in \mathfrak{S}_n \setminus \{1\}} (-1)^{\ell(w)} e^{w \lambda - \lambda},
\]
or alternatively, $\mu$ is in the interior of the convex hull of the points $\mathfrak{S}_n (-k\rho)$.
Hence, we have
\[
\ch \SDT(-\infty) = \ch \left(  \varinjlim_{\lambda \in \Lambda^-} \SDT(\lambda) \otimes \T_{-\lambda} \right) = \dfrac{\displaystyle\prod_{\alpha \in \Phi^+} (1 + e^{\alpha})}{\displaystyle\prod_{\alpha \in \Phi^+} (1 - e^{\alpha})} = \ch U_q^+\bigl(\q(n)\bigr). \qedhere
\]
\end{proof}

We can prove a stronger statement, that the lowest weight vectors, and hence the $U_q\bigl(\gl(n)\bigr)$-highest weight decomposition, of $\SDT(-\infty)$ is given by~\eqref{eq:Verma_character}.

\begin{dfn}
Call the leftmost $j$-boxes of row $j$ the \defn{trivial} boxes. Let $\rd_{nt}(T)$ denote the subword of $\rd(T)$ coming from all nontrivial boxes.
Let $T$ be a marginally large tableau and $1\le i \le k\le n$.  Define an \defn{$(i,k)$-consecution} of $T$ to be a subword of the $\rd_{nt}(T)$ of the form 
\[
k\, (k-1)\, \cdots\, (i+1)\, i
\] 
such that $i \in \red_{i-1}(T)$ and $j$ is paired with $j+1$ for $i \le j < k$. 
\end{dfn}

\begin{ex}
\label{ex:consecutions}
Consider
\[
\newcommand{\bc}[1]{{\color{darkred} \mathbf{#1}}}
\newcommand{\uc}[1]{{\color{blue} \underline{#1}}}
\gyoung(55555555555,:;4444444<\bc{2}><\bc{3}><\bc{4}>,::;333<\uc{2}><\bc{1}><\uc{3}>,:::;22,::::;1)\ ,
\]
which has a $(2,3)$-consecution consisting of the underlined entries and a $(1,4)$-consecution consisting of the bold entries. Note that there is also a $(2,2)$-consecution contained in the $(2,3)$-consecution, and similarly a $(1,k)$-consecution for $k = 1,2,3$ contained in the $(1,4)$-consecution (which are also nested).
\end{ex}

\begin{prop}\label{prop:alg}
There exists a bijection $\Xi$ between subsets of $\Phi^+$ and lowest weight elements of $\SDT(-\infty)$.
Furthermore, $\Xi$ is weight preserving: for $B = \{\beta_1, \dotsc, \beta_k\} \subseteq \Phi^+$, we have $\wt\bigl( \Xi(B) \bigr) = \beta_1 + \cdots + \beta_k$.
\end{prop}

\begin{proof}
For a fixed $B \subseteq \Phi^+$, write $B = X_2 \cup \cdots \cup X_n$, where $X_j = \{ \epsilon_i - \epsilon_j \mid i < j \} \subseteq \Phi^+$.
We construct the element $\Xi(B)$ using the following algorithm.
We start with $j = 2$ and $T_1 = L^{-\infty}$ and proceed inductively on $j$.
On step $j$, we perform the following procedure.
\begin{enumerate}
\item Order $X_j$ such that $\epsilon_i - \epsilon_j \prec \epsilon_{i'} - \epsilon_j$ if and only if $i > i'$, and enumerate the elements in $X_j$ by $\beta_1$ to $\beta_{\ell_j}$.

\item\label{step:add_roots} Start with $k = 1$ and define $T_{j,0}$ by adding a trivial column of height $n + 1 - j$ to $T_{j-1}$. Proceed inductively on $k$ as follows.
  \begin{enumerate}
  \item Suppose $\beta_k = \epsilon_i - \epsilon_j$.
  \item\label{substep:top_row} Push in a trivial column of height $n + 2 - j$. Change the entry of the $(j+1)$th entry of the $j$th row\footnote{In our convention, the $j$th row of a tableau $T$ is the $j$th row from the bottom.} of $T_{j,k-1}$ to $j - k$.
  \item Define $b = j - k$.
  \item\label{substep:check} If $b = i$, then terminate this subprocess and denote the resulting tableau by $T_{j,k}$. 
  \item\label{substep:decrease}    
   \begin{enumerate}
    \item Suppose there is a $b+1$ in $\red_{b+1}(T_{j,k-1})$ in cell $\mathsf{c}$ in $T_{j,k-1}$.  If there is a $(b+1)$-box $\mathsf{c}^{\rightarrow}$ immediate to the right of $\mathsf{c}$ in $T_{j,k-1}$, then replace the $b+1$ in $\mathsf{c}^{\rightarrow}$ with $b$.  Otherwise, replace $b+1$ in cell $\mathsf{c}$ with $b$.  Set $a=b$.
    \item Otherwise, find the leftmost $(a,b)$-consecution of $T$ such that $a > i$. Replace this subword with $(a-1)\, a\, \cdots\, (b-1)$. 
    \item If there is no such $(a,b)$-consecution, set $a = b$, push in a trivial $n+1-a$ column, and replace the rightmost trivial $a$ with an $a - 1$.
   \end{enumerate}
  \item\label{substep:redefine} Redefine $b = a - 1$ and repeat from step~(\ref{substep:check}).
  \end{enumerate}

\item Set $T_j = T_{j,\ell_j}$ to be the result of the previous subprocess.
\end{enumerate}
Define $\Xi(B) = T_n$.

Note that the increasing word\footnote{Following~\cite{GJKKK14}, the increasing word of a hook word $x_1 \geq \cdots \geq x_{\ell} < y_1 < \cdots < y_m$ is $y_1 < \cdots < y_m$.} in row $j$ has size $\absval{X_j}$; in particular,
for $\absval{X_j} > 0$, the nontrivial entries of the $j$th row of $T_j$ are
\[
\gyoung(!<\Yboxdimy{16pt}>!<\Yboxdimx{45pt}><j-|X_j|>!<\Yboxdimx{14pt}>_2\hdts!<\Yboxdimx{30pt}><j-1>!<\Yboxdimx{14pt}>j)
\]
by the result of applying~(\ref{substep:top_row}).
Furthermore, there exists precisely one $(i',j)$-consecution in $T_{j,k}$ after adding the root $\beta_k = \epsilon_i - \epsilon_j$ with $i' \leq i$.

It is clear the algorithm does not fail.
So in order to show $\Xi$ is well-defined, we need to show that $\Xi(B)$ is a lowest weight element.
Therefore, we have that $\Xi(B)$ is a $I_0$-lowest weight element since the $b-1$ after applying step~(\ref{substep:decrease}) pairs with the $a' - 1$ from the previous iteration with $k - 1$ (recall from step~(\ref{substep:redefine}) that $b = a' - 1$) and in an $(a,b)$-consecution, all letters except $a$ and possibly $b$ are paired.
Note that if no such $(a,b)$-consecution exists, we can consider the rightmost trivial $a = b$ as an $(a,a)$-consecution for the above proof.
Furthermore, any $1$ will always have $2$ before it in the reading word, and so $f_{\bon} T_{j,k} = \zero$.
Hence, the map $\Xi$ is well-defined.
Moreover, $\Xi$ is weight preserving as $\wt(T_{j,k}) - \wt(T_{j,k-1}) = \beta_k$.

Now we show that $\Xi$ is invertible, and so consider a lowest weight element $T \in \SDT(-\infty)$.
Let $j$ be maximal such that the $j$th row of $T$ has a nontrivial entry.
If so such $j$ exists, then $T = L^{-\infty}$, and we terminate as this is the unique element of weight $0$ in $\SDT(-\infty)$.
Otherwise, let $i$ be minimal such that there is an $(i,j)$-consecution. 
Increase $i$ to $i'$ until the subword $i'w = i'\, (i'+1) \, \cdots \, j$ of this $(i,j)$-consecution does not contain an $(a,b)$-consecution, for some $i< a \le b < j$.
Define $\mathsf{c}$ to be the cell in $T$ containing the $i'$ if the cell immediately to its left is not an $i'$; otherwise, $\mathsf{c}$ is this cell immediately to the left.
Define $T'$ by increasing the letter in $\mathsf{c}$ and every letter in $w$ by $1$ and removing the rightmost column $[n, \dotsc, j+1, j+1]^{\top}$ and any trivial columns so that the result is marginally large.
It is straightforward to see that this is the inverse of adding the root $\beta_{\ell_j} = \epsilon_{i'} - \epsilon_j$ to $T'$ under step~(\ref{step:add_roots}).
By the definition of an $(i',j)$-consecution, $T'$ is a lowest weight element, and furthermore, the maximality of $j$ and minimality of $i$ and $i'$ ensure that the process of doing the aforementioned procedure again yields $\beta_{\ell_j-1} \prec \beta_{\ell_j}$.
Hence, the map $\Xi$ is a bijection.
\end{proof}

\begin{ex}
Consider the roots $B = \{ \epsilon_2 - \epsilon_3, \epsilon_2 - \epsilon_4, \epsilon_1 - \epsilon_4, \epsilon_1 - \epsilon_5 \}$ for $n = 5$.
At each step, we will shade the nontrivial entries and bold the changed letters.
The first step is adding $\epsilon_2 - \epsilon_3$, which results in
\[
T_2 = T_{2,1} = \gyoung(5555555,:;444444,::;333!<\fb><\ml{2}><\ml{3}>!<\fe>,:::;22,::::;1)\ .
\]
Next we add in $\epsilon_2 - \epsilon_4$, and we note that there is no $(3,3)$-consecution, so we change the rightmost (trivial) $3$ in the third row to a $2$. Hence, we have
\[
T_{3,1} = \gyoung(5555555555,:;4444444!<\fb><\ml{3}><\ml{4}>!<\fe>,::;333!<\fb><\ml{2}>23!<\fe>,:::;22,::::;1)\ .
\]
When we add in the $\epsilon_1 - \epsilon_4$, we note there is a $(2,2)$-consecution, which is inside the $(2,3)$-consecution in the third row. Hence, we split this into a $(1,1)$-consecution and a $(3,3)$-consecution, resulting in
\[
T_3 = T_{3,2} = \gyoung(55555555555,:;4444444!<\fb><\ml{2}>34!<\fe>,::;333!<\fb>2<\ml{1}>3!<\fe>,:::;22,::::;1)\ .
\]
Now when we add in $\epsilon_1 - \epsilon_5$, there are no $(a,4)$-consecutions, so we change the rightmost trivial $4$ to a $3$. Next, by Example~\ref{ex:consecutions}, there is a $(2,3)$-consecution in the third row, which then decrease to obtain
\begin{equation}
\label{eq:algorithm_ex}
\Xi(B) = T_4 = T_{4,1} = \gyoung(555555555555!<\fb><\ml{4}><\ml{5}>!<\fe>,:;4444444!<\fb><\ml{3}>234!<\fe>,::;333!<\fb><\ml{1}>1<\ml{2}>!<\fe>,:::;22,::::;1)\ .
\end{equation}

Now let us calculate one step of the inverse operation for $T_4$. There exists a $(1,5)$-consecution that is marked in bold:
\[
\gyoung(555555555555!<\fb><\ml{4}><\ml{5}>!<\fe>,:;4444444!<\fb><\ml{3}>234!<\fe>,::;333!<\fb>1<\ml{1}><\ml{2}>!<\fe>,:::;22,::::;1)\ ,
\]
and note that it does not contain the $(1,4)$-consecution. However, the cell immediately to the left of the $1$ in the $(1,5)$-consecution is also a $1$, so the letters we increase are precisely those changed in Equation~\eqref{eq:algorithm_ex}.
\end{ex}

\begin{ex}
Consider the roots $B = \{ \epsilon_1-\epsilon_3, \epsilon_2-\epsilon_5, \epsilon_1-\epsilon_5\}$.  Adding the root $\epsilon_1-\epsilon_3$ results in
\[
T_3 = T_{3,1} = \gyoung(55555555,:;4444444,::;3333!<\fb><\ml{2}><\ml{3}>!<\fe>,:::;22!<\fb><\ml{1}>!<\fe>,::::;1)\ ,
\]
and adding the root $\epsilon_2-\epsilon_5$ results in
\[
T_{5,1} = \gyoung(5555555555!<\fb><\ml{4}><\ml{5}>!<\fe>,:;44444444!<\fb><\ml{3}>!<\fe>,::;3333!<\fb><\ml{2}>23!<\fe>,:::;22!<\fb>1!<\fe>,::::;1)\ .
\]
Finally, we add $\epsilon_1 - \epsilon_5$, and in this case, when we add the $3$ to the top row, the $3$ in the subsequent row is in $\red_3(T_{5,1})$, and so we must change that to a $2$. Denote the result $T_{5,1}'$, and then the leftmost $2$ in the next row becomes unpaired in $\red_2(T_{5,1}')$, so we must additionally change that to a $1$. However, the cell immediately to its right is also a $2$, hence, we obtain
\[
\Xi(B) = T_5 = T_{5,2} = \gyoung(5555555555!<\fb><\ml{3}>45!<\fe>,:;44444444!<\fb><\ml{2}>!<\fe>,::;3333!<\fb>2<\ml{1}>3!<\fe>,:::;22!<\fb>1!<\fe>,::::;1)\ .
\]
\end{ex}

\subsection{Recovering $\SDT(\lambda)$ from $\SDT(-\infty)$}

Our construction is parallel to the $\gl(n)$-crystal construction of $B(\lambda)$ from $B(\infty)$ by essentially undoing the direct limit construction and adjusting $\varepsilon_i(b)$ to be the number of times we can apply $e_i$ before getting $\zero$. See Figure~\ref{fig:hw} for an example. However, we note that whenever $\lambda_1-\lambda_2 = 0$, we obtain a connected component that is too large as we should have $e_{\overline{1}}( L^{-\infty} \otimes r^{\vee}_{\lambda} ) = \zero$. Thus, we would require a modification to the tensor product rule, but we can obtain $\SDT(\lambda)$ when $\lambda_i > \lambda_{i+1}$ for all $i \in I_0$.

\begin{thm}
\label{thm:cutting_it_out}
Let $\lambda \in \Lambda^-$ be such that $\lambda_i > \lambda_{i+1}$ for all $i \in I_0$.
Let 
\[
\mu = \sum_{i=1}^n (\lambda_i - k) \epsilon_{n+1-i} \text{ for some } k \geq -\min \{ \lambda_i \mid i \in I_0 \}.
\]
The connected component $\C$ of $\SDT(-\infty) \otimes \R_{\mu}^{\vee}$ generated by $L^{-\infty} \otimes r_{\mu}^{\vee}$, using the tensor product rule in Equation~\eqref{eq:tensor}, is isomorphic to $\SDT(\lambda)$ as $\q(n)$-crystals.
\end{thm}

\begin{proof}
We define a map $\psi \colon \C \longrightarrow \SDT(\lambda)$ by $\psi(T \otimes r^{\vee}_{\mu}) = E(T)$, where $E(T)$ is the semistandard decomposition tableau formed by pushing in/out trivial columns. The tensor product rule in Equation~\eqref{eq:tensor} says that $e_{\bon}(T \otimes r^{\vee}_{\mu}) = (e_{\bon} T) \otimes r^{\vee}_{\mu}$ and similarly for $f_{\bon}$. Moreover, we have in fact shown in Lemma~\ref{lem:up_inclusions} that $E$ commutes with all $e_{\bon}$ and $f_{\bon}$ crystal operators, and hence, $e_{\bon}$ and $f_{\bon}$ commute with $\psi$.
Next, we note that from the $\gl(n)$-case, we have the $I_0$-component generated by $L^{-\infty} \otimes r_{\mu}$ is isomorphic to the $I_0$-component of $L^{\lambda}$ and $\psi$ is a crystal isomorphism when restricted to these $I_0$-components.
Thus, since both $\C$ and $\SDT(\lambda)$ are $\q(n)$-crystals and the computation of $\red_i(T)$, we have that $\psi$ commutes with $e_i$ and $f_i$ for $i \in I_0$.
Hence, $\psi$ is the desired crystal isomorphism.
\end{proof}

Let us remark on the $\mu$ in Theorem~\ref{thm:cutting_it_out}. Recall that tensoring $B(\lambda)$ with the determinant representation of $\q(n)$/$\gl(n)$ (\textit{i.e.}, the one-dimensional representation of weight $(1, \dotsc, 1)$), does not change the crystal structure of $B(\lambda)$ up to a shift in the weight by $(1,\dotsc,1)$. Therefore, to remove this dependence of the determinant representation, we could modify the condition $\wt\bigl(\psi(b)\bigr) = \wt(b)$ of a crystal morphism to instead be $\wt_i\bigl(\psi(b)\bigr) = \wt_i(b)$ for all $i \in I_0$. All of the theory given above holds under this modification as our tensor product rule does not change under this shift in weight (unlike that given in~\cite{GJKKK14}). This is analogous to considering the $\mathfrak{sl}(n)$-crystals instead of $\gl(n)$-crystals.

The condition $\lambda_i > \lambda_{i+1}$ ensures that the weight $\lambda$ is not on a wall of the dominant chamber.  Furthermore, when $\lambda_i = \lambda_{i+1} > 0$, the corresponding irreducible representation is not finite-dimensional \cite[Theorem 2.18]{CW12}.  Since the crystal structure does not change when shifting the weight by $(1,1,\dots,1)$ but the representation-theoretic information does (the dimensionality changes), we require all parts to be not equal to preserve finite-dimensionality for the corresponding irreducible representation.

More concretely, if $\nu = \lambda + k (1^n)$ for some $k \in \ZZ$, then $\R_\lambda^\vee \iso \R_\nu^\vee$ and $B(\lambda) \iso B(\nu)$ as $\q(n)$-crystals under this broader definition of a crystal isomorphism. Therefore, our choice of $\mu$ in Theorem~\ref{thm:cutting_it_out} is so that we have $\mu \in \ZZ^n_{\geq 0}$, but it does not change the resulting crystal graphs, $\varepsilon_i$, and $\varphi_i$.

\begin{ex}
We consider $\q(3)$ and $\lambda = -3\epsilon_1 - \epsilon_2$. Then we can take $\mu = 3\epsilon_3 + \epsilon_2 - k(\epsilon_1 + \epsilon_2 + \epsilon_3)$ for some $k \geq 3$. Thus, when we take the connected component generated by $L^{-\infty} \otimes r^{\vee}_{\mu}$, we obtain the crystal graph in Figure~\ref{fig:hw} (taken explicitly with $k=3$, so $\mu = 0\epsilon_3 - 2 \epsilon_2 - 3\epsilon_1$). Thus for $k=3$, we obtain a crystal isomorphic to $\SDT(\lambda)$ given by~\cite[Figure~1]{GJKKK14}.
\end{ex}

\begin{figure}[t]
\[
\begin{tikzpicture}[>=latex,line join=bevel,every node/.style={scale=.5},yscale=.55,xscale=0.45]
\node (node_22) at (172.0bp,21.5bp) [draw,draw=none] {${\def\lr#1{\multicolumn{1}{|@{\hspace{.6ex}}c@{\hspace{.6ex}}|}{\raisebox{-.3ex}{$#1$}}}\raisebox{-.6ex}{$\begin{array}[b]{*{3}c}\cline{1-3}\lr{3}&\lr{3}&\lr{3}\\\cline{1-3}&\lr{2}&\lr{2}\\\cline{2-3}&&\lr{1}\\\cline{3-3}\end{array}$}} \otimes {r^\vee_{-3\epsilon_{1} - 2\epsilon_{2}}}$};
  \node (node_23) at (101.0bp,118.5bp) [draw,draw=none] {${\def\lr#1{\multicolumn{1}{|@{\hspace{.6ex}}c@{\hspace{.6ex}}|}{\raisebox{-.3ex}{$#1$}}}\raisebox{-.6ex}{$\begin{array}[b]{*{4}c}\cline{1-4}\lr{3}&\lr{3}&\lr{3}&\lr{3}\\\cline{1-4}&\lr{2}&\lr{2}&\lr{1}\\\cline{2-4}&&\lr{1}\\\cline{3-3}\end{array}$}} \otimes {r^\vee_{-3\epsilon_{1} - 2\epsilon_{2}}}$};
  \node (node_20) at (454.0bp,118.5bp) [draw,draw=none] {${\def\lr#1{\multicolumn{1}{|@{\hspace{.6ex}}c@{\hspace{.6ex}}|}{\raisebox{-.3ex}{$#1$}}}\raisebox{-.6ex}{$\begin{array}[b]{*{5}c}\cline{1-5}\lr{3}&\lr{3}&\lr{3}&\lr{2}&\lr{3}\\\cline{1-5}&\lr{2}&\lr{2}\\\cline{2-3}&&\lr{1}\\\cline{3-3}\end{array}$}} \otimes {r^\vee_{-3\epsilon_{1} - 2\epsilon_{2}}}$};
  \node (node_21) at (635.0bp,217.5bp) [draw,draw=none] {${\def\lr#1{\multicolumn{1}{|@{\hspace{.6ex}}c@{\hspace{.6ex}}|}{\raisebox{-.3ex}{$#1$}}}\raisebox{-.6ex}{$\begin{array}[b]{*{6}c}\cline{1-6}\lr{3}&\lr{3}&\lr{3}&\lr{3}&\lr{2}&\lr{3}\\\cline{1-6}&\lr{2}&\lr{2}&\lr{1}\\\cline{2-4}&&\lr{1}\\\cline{3-3}\end{array}$}} \otimes {r^\vee_{-3\epsilon_{1} - 2\epsilon_{2}}}$};
  \node (node_9) at (454.0bp,316.5bp) [draw,draw=none] {${\def\lr#1{\multicolumn{1}{|@{\hspace{.6ex}}c@{\hspace{.6ex}}|}{\raisebox{-.3ex}{$#1$}}}\raisebox{-.6ex}{$\begin{array}[b]{*{5}c}\cline{1-5}\lr{3}&\lr{3}&\lr{3}&\lr{1}&\lr{2}\\\cline{1-5}&\lr{2}&\lr{2}\\\cline{2-3}&&\lr{1}\\\cline{3-3}\end{array}$}} \otimes {r^\vee_{-3\epsilon_{1} - 2\epsilon_{2}}}$};
  \node (node_8) at (491.0bp,514.5bp) [draw,draw=none] {${\def\lr#1{\multicolumn{1}{|@{\hspace{.6ex}}c@{\hspace{.6ex}}|}{\raisebox{-.3ex}{$#1$}}}\raisebox{-.6ex}{$\begin{array}[b]{*{7}c}\cline{1-7}\lr{3}&\lr{3}&\lr{3}&\lr{3}&\lr{2}&\lr{1}&\lr{2}\\\cline{1-7}&\lr{2}&\lr{2}&\lr{1}\\\cline{2-4}&&\lr{1}\\\cline{3-3}\end{array}$}} \otimes {r^\vee_{-3\epsilon_{1} - 2\epsilon_{2}}}$};
  \node (node_7) at (196.0bp,316.5bp) [draw,draw=none] {${\def\lr#1{\multicolumn{1}{|@{\hspace{.6ex}}c@{\hspace{.6ex}}|}{\raisebox{-.3ex}{$#1$}}}\raisebox{-.6ex}{$\begin{array}[b]{*{5}c}\cline{1-5}\lr{3}&\lr{3}&\lr{3}&\lr{3}&\lr{1}\\\cline{1-5}&\lr{2}&\lr{2}&\lr{1}\\\cline{2-4}&&\lr{1}\\\cline{3-3}\end{array}$}} \otimes {r^\vee_{-3\epsilon_{1} - 2\epsilon_{2}}}$};
  \node (node_6) at (200.0bp,217.5bp) [draw,draw=none] {${\def\lr#1{\multicolumn{1}{|@{\hspace{.6ex}}c@{\hspace{.6ex}}|}{\raisebox{-.3ex}{$#1$}}}\raisebox{-.6ex}{$\begin{array}[b]{*{4}c}\cline{1-4}\lr{3}&\lr{3}&\lr{3}&\lr{1}\\\cline{1-4}&\lr{2}&\lr{2}\\\cline{2-3}&&\lr{1}\\\cline{3-3}\end{array}$}} \otimes {r^\vee_{-3\epsilon_{1} - 2\epsilon_{2}}}$};
  \node (node_5) at (67.0bp,415.5bp) [draw,draw=none] {${\def\lr#1{\multicolumn{1}{|@{\hspace{.6ex}}c@{\hspace{.6ex}}|}{\raisebox{-.3ex}{$#1$}}}\raisebox{-.6ex}{$\begin{array}[b]{*{6}c}\cline{1-6}\lr{3}&\lr{3}&\lr{3}&\lr{3}&\lr{2}&\lr{1}\\\cline{1-6}&\lr{2}&\lr{2}&\lr{1}\\\cline{2-4}&&\lr{1}\\\cline{3-3}\end{array}$}} \otimes {r^\vee_{-3\epsilon_{1} - 2\epsilon_{2}}}$};
  \node (node_4) at (325.0bp,316.5bp) [draw,draw=none] {${\def\lr#1{\multicolumn{1}{|@{\hspace{.6ex}}c@{\hspace{.6ex}}|}{\raisebox{-.3ex}{$#1$}}}\raisebox{-.6ex}{$\begin{array}[b]{*{5}c}\cline{1-5}\lr{3}&\lr{3}&\lr{3}&\lr{2}&\lr{1}\\\cline{1-5}&\lr{2}&\lr{2}\\\cline{2-3}&&\lr{1}\\\cline{3-3}\end{array}$}} \otimes {r^\vee_{-3\epsilon_{1} - 2\epsilon_{2}}}$};
  \node (node_3) at (184.0bp,514.5bp) [draw,draw=none] {${\def\lr#1{\multicolumn{1}{|@{\hspace{.6ex}}c@{\hspace{.6ex}}|}{\raisebox{-.3ex}{$#1$}}}\raisebox{-.6ex}{$\begin{array}[b]{*{7}c}\cline{1-7}\lr{3}&\lr{3}&\lr{3}&\lr{3}&\lr{2}&\lr{2}&\lr{1}\\\cline{1-7}&\lr{2}&\lr{2}&\lr{1}\\\cline{2-4}&&\lr{1}\\\cline{3-3}\end{array}$}} \otimes {r^\vee_{-3\epsilon_{1} - 2\epsilon_{2}}}$};
  \node (node_2) at (346.0bp,514.5bp) [draw,draw=none] {${\def\lr#1{\multicolumn{1}{|@{\hspace{.6ex}}c@{\hspace{.6ex}}|}{\raisebox{-.3ex}{$#1$}}}\raisebox{-.6ex}{$\begin{array}[b]{*{6}c}\cline{1-6}\lr{3}&\lr{3}&\lr{3}&\lr{3}&\lr{1}&\lr{1}\\\cline{1-6}&\lr{2}&\lr{2}&\lr{1}\\\cline{2-4}&&\lr{1}\\\cline{3-3}\end{array}$}} \otimes {r^\vee_{-3\epsilon_{1} - 2\epsilon_{2}}}$};
  \node (node_1) at (352.0bp,415.5bp) [draw,draw=none] {${\def\lr#1{\multicolumn{1}{|@{\hspace{.6ex}}c@{\hspace{.6ex}}|}{\raisebox{-.3ex}{$#1$}}}\raisebox{-.6ex}{$\begin{array}[b]{*{5}c}\cline{1-5}\lr{3}&\lr{3}&\lr{3}&\lr{1}&\lr{1}\\\cline{1-5}&\lr{2}&\lr{2}\\\cline{2-3}&&\lr{1}\\\cline{3-3}\end{array}$}} \otimes {r^\vee_{-3\epsilon_{1} - 2\epsilon_{2}}}$};
  \node (node_0) at (346.0bp,613.5bp) [draw,draw=none] {${\def\lr#1{\multicolumn{1}{|@{\hspace{.6ex}}c@{\hspace{.6ex}}|}{\raisebox{-.3ex}{$#1$}}}\raisebox{-.6ex}{$\begin{array}[b]{*{7}c}\cline{1-7}\lr{3}&\lr{3}&\lr{3}&\lr{3}&\lr{2}&\lr{1}&\lr{1}\\\cline{1-7}&\lr{2}&\lr{2}&\lr{1}\\\cline{2-4}&&\lr{1}\\\cline{3-3}\end{array}$}} \otimes {r^\vee_{-3\epsilon_{1} - 2\epsilon_{2}}}$};
  \node (node_19) at (734.0bp,316.5bp) [draw,draw=none] {${\def\lr#1{\multicolumn{1}{|@{\hspace{.6ex}}c@{\hspace{.6ex}}|}{\raisebox{-.3ex}{$#1$}}}\raisebox{-.6ex}{$\begin{array}[b]{*{7}c}\cline{1-7}\lr{3}&\lr{3}&\lr{3}&\lr{3}&\lr{2}&\lr{2}&\lr{3}\\\cline{1-7}&\lr{2}&\lr{2}&\lr{1}\\\cline{2-4}&&\lr{1}\\\cline{3-3}\end{array}$}} \otimes {r^\vee_{-3\epsilon_{1} - 2\epsilon_{2}}}$};
  \node (node_18) at (589.0bp,316.5bp) [draw,draw=none] {${\def\lr#1{\multicolumn{1}{|@{\hspace{.6ex}}c@{\hspace{.6ex}}|}{\raisebox{-.3ex}{$#1$}}}\raisebox{-.6ex}{$\begin{array}[b]{*{6}c}\cline{1-6}\lr{3}&\lr{3}&\lr{3}&\lr{3}&\lr{1}&\lr{3}\\\cline{1-6}&\lr{2}&\lr{2}&\lr{1}\\\cline{2-4}&&\lr{1}\\\cline{3-3}\end{array}$}} \otimes {r^\vee_{-3\epsilon_{1} - 2\epsilon_{2}}}$};
  \node (node_17) at (454.0bp,217.5bp) [draw,draw=none] {${\def\lr#1{\multicolumn{1}{|@{\hspace{.6ex}}c@{\hspace{.6ex}}|}{\raisebox{-.3ex}{$#1$}}}\raisebox{-.6ex}{$\begin{array}[b]{*{5}c}\cline{1-5}\lr{3}&\lr{3}&\lr{3}&\lr{1}&\lr{3}\\\cline{1-5}&\lr{2}&\lr{2}\\\cline{2-3}&&\lr{1}\\\cline{3-3}\end{array}$}} \otimes {r^\vee_{-3\epsilon_{1} - 2\epsilon_{2}}}$};
  \node (node_16) at (734.0bp,415.5bp) [draw,draw=none] {${\def\lr#1{\multicolumn{1}{|@{\hspace{.6ex}}c@{\hspace{.6ex}}|}{\raisebox{-.3ex}{$#1$}}}\raisebox{-.6ex}{$\begin{array}[b]{*{7}c}\cline{1-7}\lr{3}&\lr{3}&\lr{3}&\lr{3}&\lr{2}&\lr{1}&\lr{3}\\\cline{1-7}&\lr{2}&\lr{2}&\lr{1}\\\cline{2-4}&&\lr{1}\\\cline{3-3}\end{array}$}} \otimes {r^\vee_{-3\epsilon_{1} - 2\epsilon_{2}}}$};
  \node (node_15) at (76.0bp,217.5bp) [draw,draw=none] {${\def\lr#1{\multicolumn{1}{|@{\hspace{.6ex}}c@{\hspace{.6ex}}|}{\raisebox{-.3ex}{$#1$}}}\raisebox{-.6ex}{$\begin{array}[b]{*{5}c}\cline{1-5}\lr{3}&\lr{3}&\lr{3}&\lr{3}&\lr{2}\\\cline{1-5}&\lr{2}&\lr{2}&\lr{1}\\\cline{2-4}&&\lr{1}\\\cline{3-3}\end{array}$}} \otimes {r^\vee_{-3\epsilon_{1} - 2\epsilon_{2}}}$};
  \node (node_14) at (223.0bp,118.5bp) [draw,draw=none] {${\def\lr#1{\multicolumn{1}{|@{\hspace{.6ex}}c@{\hspace{.6ex}}|}{\raisebox{-.3ex}{$#1$}}}\raisebox{-.6ex}{$\begin{array}[b]{*{4}c}\cline{1-4}\lr{3}&\lr{3}&\lr{3}&\lr{2}\\\cline{1-4}&\lr{2}&\lr{2}\\\cline{2-3}&&\lr{1}\\\cline{3-3}\end{array}$}} \otimes {r^\vee_{-3\epsilon_{1} - 2\epsilon_{2}}}$};
  \node (node_13) at (61.0bp,316.5bp) [draw,draw=none] {${\def\lr#1{\multicolumn{1}{|@{\hspace{.6ex}}c@{\hspace{.6ex}}|}{\raisebox{-.3ex}{$#1$}}}\raisebox{-.6ex}{$\begin{array}[b]{*{6}c}\cline{1-6}\lr{3}&\lr{3}&\lr{3}&\lr{3}&\lr{2}&\lr{2}\\\cline{1-6}&\lr{2}&\lr{2}&\lr{1}\\\cline{2-4}&&\lr{1}\\\cline{3-3}\end{array}$}} \otimes {r^\vee_{-3\epsilon_{1} - 2\epsilon_{2}}}$};
  \node (node_12) at (324.0bp,217.5bp) [draw,draw=none] {${\def\lr#1{\multicolumn{1}{|@{\hspace{.6ex}}c@{\hspace{.6ex}}|}{\raisebox{-.3ex}{$#1$}}}\raisebox{-.6ex}{$\begin{array}[b]{*{5}c}\cline{1-5}\lr{3}&\lr{3}&\lr{3}&\lr{2}&\lr{2}\\\cline{1-5}&\lr{2}&\lr{2}\\\cline{2-3}&&\lr{1}\\\cline{3-3}\end{array}$}} \otimes {r^\vee_{-3\epsilon_{1} - 2\epsilon_{2}}}$};
  \node (node_11) at (212.0bp,415.5bp) [draw,draw=none] {${\def\lr#1{\multicolumn{1}{|@{\hspace{.6ex}}c@{\hspace{.6ex}}|}{\raisebox{-.3ex}{$#1$}}}\raisebox{-.6ex}{$\begin{array}[b]{*{7}c}\cline{1-7}\lr{3}&\lr{3}&\lr{3}&\lr{3}&\lr{2}&\lr{2}&\lr{2}\\\cline{1-7}&\lr{2}&\lr{2}&\lr{1}\\\cline{2-4}&&\lr{1}\\\cline{3-3}\end{array}$}} \otimes {r^\vee_{-3\epsilon_{1} - 2\epsilon_{2}}}$};
  \node (node_10) at (488.0bp,415.5bp) [draw,draw=none] {${\def\lr#1{\multicolumn{1}{|@{\hspace{.6ex}}c@{\hspace{.6ex}}|}{\raisebox{-.3ex}{$#1$}}}\raisebox{-.6ex}{$\begin{array}[b]{*{6}c}\cline{1-6}\lr{3}&\lr{3}&\lr{3}&\lr{3}&\lr{1}&\lr{2}\\\cline{1-6}&\lr{2}&\lr{2}&\lr{1}\\\cline{2-4}&&\lr{1}\\\cline{3-3}\end{array}$}} \otimes {r^\vee_{-3\epsilon_{1} - 2\epsilon_{2}}}$};
  \draw [red,->] (node_0) ..controls (346.0bp,578.68bp) and (346.0bp,561.22bp)  .. (node_2);
  \definecolor{strokecol}{rgb}{0.0,0.0,0.0};
  \pgfsetstrokecolor{strokecol}
  \draw (354.5bp,564.0bp) node {$2$};
  \draw [red,->] (node_8) ..controls (489.95bp,479.68bp) and (489.41bp,462.22bp)  .. (node_10);
  \draw (498.5bp,465.0bp) node {$2$};
  \draw [blue,->] (node_7) ..controls (197.39bp,281.68bp) and (198.11bp,264.22bp)  .. (node_6);
  \draw (207.5bp,267.0bp) node {$1$};
  \draw [blue,->] (node_10) ..controls (476.1bp,380.55bp) and (469.91bp,362.88bp)  .. (node_9);
  \draw (483.5bp,366.0bp) node {$1$};
  \draw [red,->] (node_10) ..controls (524.15bp,379.78bp) and (543.89bp,360.83bp)  .. (node_18);
  \draw (556.5bp,366.0bp) node {$2$};
  \draw [dashed,blue,->] (node_18) ..controls (605.16bp,281.43bp) and (613.64bp,263.54bp)  .. (node_21);
  \draw (625.5bp,267.0bp) node {$\overline{1}$};
  \draw [blue,->] (node_5) ..controls (70.786bp,382.64bp) and (71.445bp,368.53bp)  .. (70.0bp,356.0bp) .. controls (69.704bp,353.43bp) and (69.301bp,350.79bp)  .. (node_13);
  \draw (79.5bp,366.0bp) node {$1$};
  \draw [dashed,blue,->] (node_5) ..controls (52.745bp,388.28bp) and (50.345bp,382.05bp)  .. (49.0bp,376.0bp) .. controls (46.974bp,366.88bp) and (47.86bp,356.94bp)  .. (node_13);
  \draw (57.5bp,366.0bp) node {$\overline{1}$};
  \draw [blue,->] (node_0) ..controls (286.94bp,577.14bp) and (253.47bp,557.09bp)  .. (node_3);
  \draw (288.5bp,564.0bp) node {$1$};
  \draw [red,->] (node_5) ..controls (113.68bp,379.4bp) and (139.76bp,359.79bp)  .. (node_7);
  \draw (152.5bp,366.0bp) node {$2$};
  \draw [blue,->] (node_23) ..controls (126.36bp,83.566bp) and (139.46bp,66.039bp)  .. (node_22);
  \draw (151.5bp,70.0bp) node {$1$};
  \draw [red,->] (node_11) ..controls (191.2bp,381.52bp) and (179.06bp,366.13bp)  .. (165.0bp,356.0bp) .. controls (158.72bp,351.48bp) and (146.13bp,346.0bp)  .. (node_13);
  \draw (194.5bp,366.0bp) node {$2$};
  \draw [red,->] (node_19) ..controls (698.57bp,280.78bp) and (679.22bp,261.83bp)  .. (node_21);
  \draw (701.5bp,267.0bp) node {$2$};
  \draw [red,->] (node_3) ..controls (141.81bp,478.53bp) and (118.43bp,459.14bp)  .. (node_5);
  \draw (144.5bp,465.0bp) node {$2$};
  \draw [blue,->] (node_18) ..controls (540.14bp,280.4bp) and (512.86bp,260.79bp)  .. (node_17);
  \draw (542.5bp,267.0bp) node {$1$};
  \draw [red,->] (node_12) ..controls (287.85bp,181.78bp) and (268.11bp,162.83bp)  .. (node_14);
  \draw (291.5bp,168.0bp) node {$2$};
  \draw [red,->] (node_9) ..controls (454.0bp,281.68bp) and (454.0bp,264.22bp)  .. (node_17);
  \draw (462.5bp,267.0bp) node {$2$};
  \draw [blue,->] (node_17) ..controls (438.96bp,190.34bp) and (436.42bp,184.1bp)  .. (435.0bp,178.0bp) .. controls (432.98bp,169.34bp) and (432.98bp,166.66bp)  .. (435.0bp,158.0bp) .. controls (435.67bp,155.14bp) and (436.58bp,152.25bp)  .. (node_20);
  \draw (443.5bp,168.0bp) node {$1$};
  \draw [dashed,blue,->] (node_17) ..controls (454.0bp,182.68bp) and (454.0bp,165.22bp)  .. (node_20);
  \draw (462.5bp,168.0bp) node {$\overline{1}$};
  \draw [red,->] (node_4) ..controls (279.85bp,280.46bp) and (254.72bp,260.96bp)  .. (node_6);
  \draw (282.5bp,267.0bp) node {$2$};
  \draw [dashed,blue,->] (node_7) ..controls (152.73bp,280.53bp) and (128.75bp,261.14bp)  .. (node_15);
  \draw (155.5bp,267.0bp) node {$\overline{1}$};
  \draw [dashed,blue,->] (node_1) ..controls (388.51bp,379.78bp) and (408.44bp,360.83bp)  .. (node_9);
  \draw (421.5bp,366.0bp) node {$\overline{1}$};
  \draw [blue,->] (node_6) ..controls (208.05bp,182.55bp) and (212.24bp,164.88bp)  .. (node_14);
  \draw (222.5bp,168.0bp) node {$1$};
  \draw [dashed,blue,->] (node_6) ..controls (186.98bp,184.54bp) and (184.39bp,170.18bp)  .. (189.0bp,158.0bp) .. controls (190.27bp,154.65bp) and (191.98bp,151.39bp)  .. (node_14);
  \draw (197.5bp,168.0bp) node {$\overline{1}$};
  \draw [red,->] (node_14) ..controls (204.92bp,83.815bp) and (195.72bp,66.693bp)  .. (node_22);
  \draw (211.5bp,70.0bp) node {$2$};
  \draw [dashed,blue,->] (node_0) ..controls (398.67bp,577.27bp) and (428.3bp,557.44bp)  .. (node_8);
  \draw (440.5bp,564.0bp) node {$\overline{1}$};
  \draw [red,->] (node_13) ..controls (66.229bp,281.68bp) and (68.93bp,264.22bp)  .. (node_15);
  \draw (78.5bp,267.0bp) node {$2$};
  \draw [blue,->] (node_16) ..controls (718.96bp,388.34bp) and (716.42bp,382.1bp)  .. (715.0bp,376.0bp) .. controls (712.98bp,367.34bp) and (712.98bp,364.66bp)  .. (715.0bp,356.0bp) .. controls (715.67bp,353.14bp) and (716.58bp,350.25bp)  .. (node_19);
  \draw (723.5bp,366.0bp) node {$1$};
  \draw [dashed,blue,->] (node_16) ..controls (734.0bp,380.68bp) and (734.0bp,363.22bp)  .. (node_19);
  \draw (742.5bp,366.0bp) node {$\overline{1}$};
  \draw [blue,->] (node_4) ..controls (324.65bp,281.68bp) and (324.47bp,264.22bp)  .. (node_12);
  \draw (333.5bp,267.0bp) node {$1$};
  \draw [dashed,blue,->] (node_4) ..controls (308.25bp,289.17bp) and (305.54bp,283.05bp)  .. (304.0bp,277.0bp) .. controls (301.8bp,268.39bp) and (301.89bp,265.64bp)  .. (304.0bp,257.0bp) .. controls (304.69bp,254.19bp) and (305.62bp,251.35bp)  .. (node_12);
  \draw (312.5bp,267.0bp) node {$\overline{1}$};
  \draw [blue,->] (node_1) ..controls (342.55bp,380.55bp) and (337.63bp,362.88bp)  .. (node_4);
  \draw (349.5bp,366.0bp) node {$1$};
  \draw [blue,->] (node_3) ..controls (192.4bp,487.02bp) and (194.31bp,480.78bp)  .. (196.0bp,475.0bp) .. controls (198.66bp,465.9bp) and (201.43bp,455.95bp)  .. (node_11);
  \draw (210.5bp,465.0bp) node {$1$};
  \draw [dashed,blue,->] (node_3) ..controls (172.27bp,481.37bp) and (170.16bp,467.01bp)  .. (175.0bp,455.0bp) .. controls (176.4bp,451.53bp) and (178.28bp,448.19bp)  .. (node_11);
  \draw (183.5bp,465.0bp) node {$\overline{1}$};
  \draw [blue,->] (node_2) ..controls (348.09bp,479.68bp) and (349.17bp,462.22bp)  .. (node_1);
  \draw (358.5bp,465.0bp) node {$1$};
  \draw [dashed,blue,->] (node_2) ..controls (397.48bp,478.33bp) and (426.35bp,458.62bp)  .. (node_10);
  \draw (438.5bp,465.0bp) node {$\overline{1}$};
  \draw [red,->] (node_15) ..controls (84.749bp,182.55bp) and (89.304bp,164.88bp)  .. (node_23);
  \draw (99.5bp,168.0bp) node {$2$};
  \draw [dashed,blue,->] (node_23) ..controls (98.0bp,67.5bp) and (120.0bp,52.5bp) .. (node_22);
  \draw (117.0bp,67.5bp) node {$\overline{1}$};
\end{tikzpicture}
\]
\caption{The $\q(3)$-crystal $\SDT(\lambda)$ with $\lambda = -3\epsilon_1 - \epsilon_2$ created using \textsc{SageMath}~\cite{sage}.}
\label{fig:hw}
\end{figure}
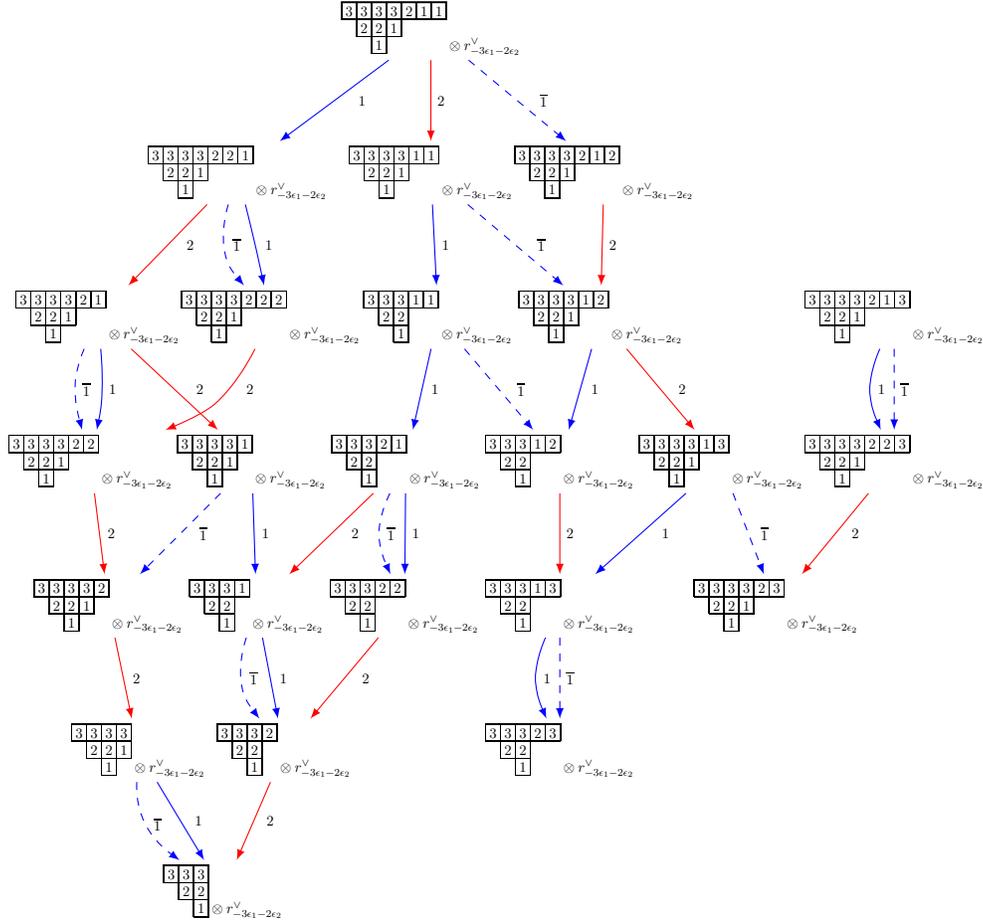

Let us discuss how to extend Theorem~\ref{thm:cutting_it_out} to more general cases.
Consider the examples in Figure~\ref{fig:more_cutting}. For $\lambda = -\epsilon_1-\epsilon_2$, we note that the connected component we obtain after also setting $e_{\overline{1}} (L^{-\infty} \otimes r^{\vee}_{\lambda}) = \zero$ is isomorphic to $\SDT(\lambda)$. Therefore, a suitably modified tensor product rule should yield $\SDT(\lambda)$ when $\lambda$ may also contain zero entries. Furthermore, we would expect a modified tensor product rule to yield dual representations. For instance, if we consider $\lambda = -\epsilon_1$, note that after setting $e_{\overline{1}} (e_2 e_1 L^{-\infty} \otimes r^{\vee}_{\lambda}) = \zero$, we would obtain the dual version of $\SDT(-\epsilon_1 - \epsilon_2)$.

\begin{figure}
\[
\begin{tikzpicture}[>=latex,line join=bevel,every node/.style={scale=.5},xscale=0.5,yscale=.5]
\node (node_7) at (179.5bp,120.5bp) [draw,draw=none] {${\def\lr#1{\multicolumn{1}{|@{\hspace{.6ex}}c@{\hspace{.6ex}}|}{\raisebox{-.3ex}{$#1$}}}\raisebox{-.6ex}{$\begin{array}[b]{*{4}c}\cline{1-4}\lr{3}&\lr{3}&\lr{3}&\lr{3}\\\cline{1-4}&\lr{2}&\lr{2}&\lr{1}\\\cline{2-4}&&\lr{1}\\\cline{3-3}\end{array}$}} \otimes {r^{\vee}_{-\epsilon_{1} - \epsilon_2}}$};
  \node (node_6) at (132.5bp,21.5bp) [draw,draw=none] {${\def\lr#1{\multicolumn{1}{|@{\hspace{.6ex}}c@{\hspace{.6ex}}|}{\raisebox{-.3ex}{$#1$}}}\raisebox{-.6ex}{$\begin{array}[b]{*{3}c}\cline{1-3}\lr{3}&\lr{3}&\lr{3}\\\cline{1-3}&\lr{2}&\lr{2}\\\cline{2-3}&&\lr{1}\\\cline{3-3}\end{array}$}} \otimes {r^{\vee}_{-\epsilon_{1} - \epsilon_2}}$};
  \node (node_5) at (179.5bp,219.5bp) [draw,draw=none] {${\def\lr#1{\multicolumn{1}{|@{\hspace{.6ex}}c@{\hspace{.6ex}}|}{\raisebox{-.3ex}{$#1$}}}\raisebox{-.6ex}{$\begin{array}[b]{*{5}c}\cline{1-5}\lr{3}&\lr{3}&\lr{3}&\lr{3}&\lr{2}\\\cline{1-5}&\lr{2}&\lr{2}&\lr{1}\\\cline{2-4}&&\lr{1}\\\cline{3-3}\end{array}$}} \otimes {r^{\vee}_{-\epsilon_{1} - \epsilon_2}}$};
  \node (node_4) at (68.5bp,120.5bp) [draw,draw=none] {${\def\lr#1{\multicolumn{1}{|@{\hspace{.6ex}}c@{\hspace{.6ex}}|}{\raisebox{-.3ex}{$#1$}}}\raisebox{-.6ex}{$\begin{array}[b]{*{4}c}\cline{1-4}\lr{3}&\lr{3}&\lr{3}&\lr{2}\\\cline{1-4}&\lr{2}&\lr{2}\\\cline{2-3}&&\lr{1}\\\cline{3-3}\end{array}$}} \otimes {r^{\vee}_{-\epsilon_{1} - \epsilon_2}}$};
  \node (node_3) at (232.5bp,318.5bp) [draw,draw=none] {${\def\lr#1{\multicolumn{1}{|@{\hspace{.6ex}}c@{\hspace{.6ex}}|}{\raisebox{-.3ex}{$#1$}}}\raisebox{-.6ex}{$\begin{array}[b]{*{6}c}\cline{1-6}\lr{3}&\lr{3}&\lr{3}&\lr{3}&\lr{2}&\lr{2}\\\cline{1-6}&\lr{2}&\lr{2}&\lr{1}\\\cline{2-4}&&\lr{1}\\\cline{3-3}\end{array}$}} \otimes {r^{\vee}_{-\epsilon_{1} - \epsilon_2}}$};
  \node (node_2) at (105.5bp,318.5bp) [draw,draw=none] {${\def\lr#1{\multicolumn{1}{|@{\hspace{.6ex}}c@{\hspace{.6ex}}|}{\raisebox{-.3ex}{$#1$}}}\raisebox{-.6ex}{$\begin{array}[b]{*{5}c}\cline{1-5}\lr{3}&\lr{3}&\lr{3}&\lr{3}&\lr{1}\\\cline{1-5}&\lr{2}&\lr{2}&\lr{1}\\\cline{2-4}&&\lr{1}\\\cline{3-3}\end{array}$}} \otimes {r^{\vee}_{-\epsilon_{1} - \epsilon_2}}$};
  \node (node_1) at (46.5bp,219.5bp) [draw,draw=none] {${\def\lr#1{\multicolumn{1}{|@{\hspace{.6ex}}c@{\hspace{.6ex}}|}{\raisebox{-.3ex}{$#1$}}}\raisebox{-.6ex}{$\begin{array}[b]{*{4}c}\cline{1-4}\lr{3}&\lr{3}&\lr{3}&\lr{1}\\\cline{1-4}&\lr{2}&\lr{2}\\\cline{2-3}&&\lr{1}\\\cline{3-3}\end{array}$}} \otimes {r^{\vee}_{-\epsilon_{1} - \epsilon_2}}$};
  \node (node_0) at (179.5bp,417.5bp) [draw,draw=none] {${\def\lr#1{\multicolumn{1}{|@{\hspace{.6ex}}c@{\hspace{.6ex}}|}{\raisebox{-.3ex}{$#1$}}}\raisebox{-.6ex}{$\begin{array}[b]{*{6}c}\cline{1-6}\lr{3}&\lr{3}&\lr{3}&\lr{3}&\lr{2}&\lr{1}\\\cline{1-6}&\lr{2}&\lr{2}&\lr{1}\\\cline{2-4}&&\lr{1}\\\cline{3-3}\end{array}$}} \otimes {r^{\vee}_{-\epsilon_{1} - \epsilon_2}}$};
  \draw [blue,->] (node_1) ..controls (36.146bp,186.35bp) and (34.336bp,172.12bp)  .. (38.5bp,160.0bp) .. controls (39.574bp,156.87bp) and (41.017bp,153.81bp)  .. (node_4);
  \definecolor{strokecol}{rgb}{0.0,0.0,0.0};
  \pgfsetstrokecolor{strokecol}
  \draw (47.0bp,170.0bp) node {$1$};
  \draw [dashed,blue,->] (node_1) ..controls (52.742bp,191.99bp) and (54.188bp,185.76bp)  .. (55.5bp,180.0bp) .. controls (57.561bp,170.96bp) and (59.763bp,161.12bp)  .. (node_4);
  \draw (69.0bp,170.0bp) node {$\overline{1}$};
  \draw [dashed,blue,->] (node_0) ..controls (198.19bp,382.3bp) and (208.08bp,364.2bp)  .. (node_3);
  \draw (219.0bp,368.0bp) node {$\overline{1}$};
  \draw [red,->] (node_0) ..controls (153.21bp,382.04bp) and (139.08bp,363.51bp)  .. (node_2);
  \draw (157.0bp,368.0bp) node {$2$};
  \draw [red,->] (node_5) ..controls (179.5bp,184.68bp) and (179.5bp,167.22bp)  .. (node_7);
  \draw (188.0bp,170.0bp) node {$2$};
  \draw [red,->] (node_4) ..controls (91.152bp,85.168bp) and (103.23bp,66.858bp)  .. (node_6);
  \draw (115.0bp,71.0bp) node {$2$};
  \draw [dashed,blue,->] (node_7) ..controls (162.99bp,85.425bp) and (154.32bp,67.54bp)  .. (node_6);
  \draw (168.0bp,71.0bp) node {$\overline{1}$};
  \draw [red,->] (node_3) ..controls (213.81bp,283.3bp) and (203.92bp,265.2bp)  .. (node_5);
  \draw (219.0bp,269.0bp) node {$2$};
  \draw [dashed,blue,->] (node_2) ..controls (131.79bp,283.04bp) and (145.92bp,264.51bp)  .. (node_5);
  \draw (157.0bp,269.0bp) node {$\overline{1}$};
\end{tikzpicture}
\qquad\qquad
\begin{tikzpicture}[>=latex,line join=bevel,every node/.style={scale=.5},xscale=0.5,yscale=.5]
\node (node_5) at (159.5bp,120.5bp) [draw,draw=none] {${\def\lr#1{\multicolumn{1}{|@{\hspace{.6ex}}c@{\hspace{.6ex}}|}{\raisebox{-.3ex}{$#1$}}}\raisebox{-.6ex}{$\begin{array}[b]{*{4}c}\cline{1-4}\lr{3}&\lr{3}&\lr{3}&\lr{3}\\\cline{1-4}&\lr{2}&\lr{2}&\lr{1}\\\cline{2-4}&&\lr{1}\\\cline{3-3}\end{array}$}} \otimes {r^{\vee}_{-\epsilon_1}}$};
  \node (node_4) at (159.5bp,21.5bp) [draw,draw=none] {${\def\lr#1{\multicolumn{1}{|@{\hspace{.6ex}}c@{\hspace{.6ex}}|}{\raisebox{-.3ex}{$#1$}}}\raisebox{-.6ex}{$\begin{array}[b]{*{3}c}\cline{1-3}\lr{3}&\lr{3}&\lr{3}\\\cline{1-3}&\lr{2}&\lr{2}\\\cline{2-3}&&\lr{1}\\\cline{3-3}\end{array}$}} \otimes {r^{\vee}_{-\epsilon_1}}$};
  \node (node_3) at (159.5bp,219.5bp) [draw,draw=none] {${\def\lr#1{\multicolumn{1}{|@{\hspace{.6ex}}c@{\hspace{.6ex}}|}{\raisebox{-.3ex}{$#1$}}}\raisebox{-.6ex}{$\begin{array}[b]{*{5}c}\cline{1-5}\lr{3}&\lr{3}&\lr{3}&\lr{3}&\lr{2}\\\cline{1-5}&\lr{2}&\lr{2}&\lr{1}\\\cline{2-4}&&\lr{1}\\\cline{3-3}\end{array}$}} \otimes {r^{\vee}_{-\epsilon_1}}$};
  \node (node_2) at (39.5bp,120.5bp) [draw,draw=none] {${\def\lr#1{\multicolumn{1}{|@{\hspace{.6ex}}c@{\hspace{.6ex}}|}{\raisebox{-.3ex}{$#1$}}}\raisebox{-.6ex}{$\begin{array}[b]{*{4}c}\cline{1-4}\lr{3}&\lr{3}&\lr{3}&\lr{2}\\\cline{1-4}&\lr{2}&\lr{2}\\\cline{2-3}&&\lr{1}\\\cline{3-3}\end{array}$}} \otimes {r^{\vee}_{-\epsilon_1}}$};
  \node (node_1) at (115.5bp,318.5bp) [draw,draw=none] {${\def\lr#1{\multicolumn{1}{|@{\hspace{.6ex}}c@{\hspace{.6ex}}|}{\raisebox{-.3ex}{$#1$}}}\raisebox{-.6ex}{$\begin{array}[b]{*{5}c}\cline{1-5}\lr{3}&\lr{3}&\lr{3}&\lr{3}&\lr{1}\\\cline{1-5}&\lr{2}&\lr{2}&\lr{1}\\\cline{2-4}&&\lr{1}\\\cline{3-3}\end{array}$}} \otimes {r^{\vee}_{-\epsilon_1}}$};
  \node (node_0) at (57.5bp,219.5bp) [draw,draw=none] {${\def\lr#1{\multicolumn{1}{|@{\hspace{.6ex}}c@{\hspace{.6ex}}|}{\raisebox{-.3ex}{$#1$}}}\raisebox{-.6ex}{$\begin{array}[b]{*{4}c}\cline{1-4}\lr{3}&\lr{3}&\lr{3}&\lr{1}\\\cline{1-4}&\lr{2}&\lr{2}\\\cline{2-3}&&\lr{1}\\\cline{3-3}\end{array}$}} \otimes {r^{\vee}_{-\epsilon_1}}$};
  \draw [dashed,blue,->] (node_1) ..controls (130.96bp,283.43bp) and (139.07bp,265.54bp)  .. (node_3);
  \definecolor{strokecol}{rgb}{0.0,0.0,0.0};
  \pgfsetstrokecolor{strokecol}
  \draw (150.0bp,269.0bp) node {$\overline{1}$};
  \draw [blue,->] (node_0) ..controls (42.457bp,192.34bp) and (39.923bp,186.1bp)  .. (38.5bp,180.0bp) .. controls (36.426bp,171.11bp) and (35.931bp,161.26bp)  .. (node_2);
  \draw (47.0bp,170.0bp) node {$1$};
  \draw [dashed,blue,->] (node_0) ..controls (58.738bp,186.65bp) and (58.198bp,172.41bp)  .. (55.5bp,160.0bp) .. controls (54.915bp,157.31bp) and (54.15bp,154.57bp)  .. (node_2);
  \draw (67.0bp,170.0bp) node {$\overline{1}$};
  \draw [blue,->] (node_1) ..controls (95.049bp,283.3bp) and (84.227bp,265.2bp)  .. (node_0);
  \draw (101.0bp,269.0bp) node {$1$};
  \draw [blue,->] (node_5) ..controls (144.46bp,93.338bp) and (141.92bp,87.097bp)  .. (140.5bp,81.0bp) .. controls (138.48bp,72.344bp) and (138.48bp,69.656bp)  .. (140.5bp,61.0bp) .. controls (141.17bp,58.142bp) and (142.08bp,55.252bp)  .. (node_4);
  \draw (149.0bp,71.0bp) node {$1$};
  \draw [dashed,blue,->] (node_5) ..controls (159.5bp,85.683bp) and (159.5bp,68.22bp)  .. (node_4);
  \draw (168.0bp,71.0bp) node {$\overline{1}$};
  \draw [red,->] (node_3) ..controls (159.5bp,184.68bp) and (159.5bp,167.22bp)  .. (node_5);
  \draw (168.0bp,170.0bp) node {$2$};
\end{tikzpicture}
\]
\caption{The $\q(3)$-crystal connected component of $\SDT(-\infty) \otimes \R_{-\epsilon_1}^\vee$ (resp.\ $\SDT(-\infty) \otimes \R_{-\epsilon_1-\epsilon_2}^\vee$) generated by $L^{-\infty} \otimes r^{\vee}_{-\epsilon_1}$ (resp.\ $L^{-\infty} \otimes r^{\vee}_{-\epsilon_1-\epsilon_2}$).}
\label{fig:more_cutting}
\end{figure}
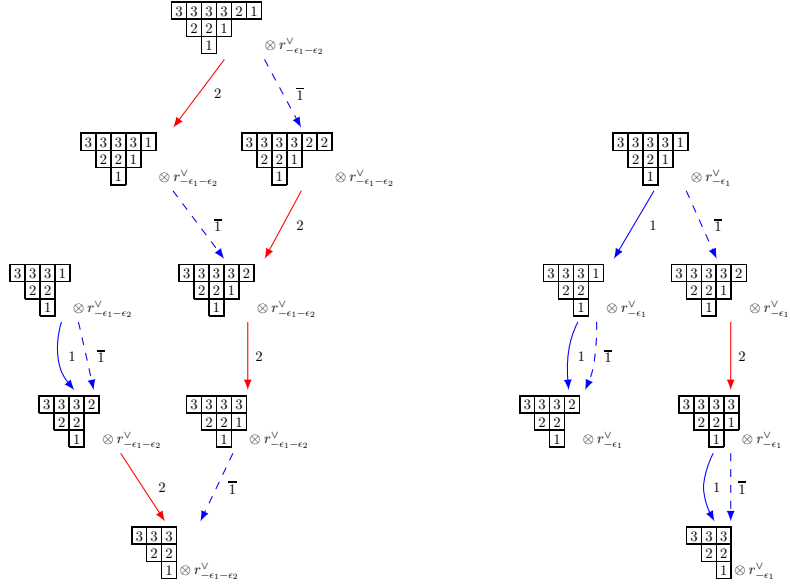

\subsection*{Acknowledgements}
T.S.\ would like to thank Central Michigan University for its hospitality during his visit in November, 2018 and January, 2020.
The authors thank the referees for helpful comments.
This work benefited from computations using \textsc{SageMath}~\cite{sage}.

\appendix
\section{\textsc{SageMath} code}

We construct Figure~\ref{fig:hw}.

\begin{lstlisting}
sage: B = crystals.infinity.Tableaux(['Q',3])
sage: e = B.weight_lattice_realization().basis()
sage: R = crystals.elementary.R(['Q',3], -3*e[0]-2*e[1], dual=True)
sage: x = tensor([B.module_generators[0], R.module_generators[0]])
sage: view(x.subcrystal(max_depth=10,index_set=[-1,1,2]))
\end{lstlisting}

In Figure~\ref{fig:hw}, we have done some mild post-processing on the resulting crystal graph to shift the basis of $\Lambda^-$ from being $(e_0, e_1, e_2)$ to $(\epsilon_0, \epsilon_1, \epsilon_2)$. The construction of Figure~\ref{fig:more_cutting} is similar.

\bibliographystyle{plain}
\bibliography{infinitycrystalqn}{}

\end{document}